\documentclass{amsart}
\usepackage{amsmath,amssymb}
\newtheorem{theorem}{Theorem}[section]
\theoremstyle{definition}
\newtheorem{definition}[theorem]{Definition}
\newtheorem{lemma}[theorem]{Lemma}

\newtheorem{corollary}[theorem]{Corollary}

\numberwithin{equation}{section}

\def\min{\operatorname{Min}}
\def\max{\operatorname{Max}}

\title {Spectrum of the 1-Laplacian and Cheeger's constant on graphs  }

\author {K.C. Chang}
\address{LMAM, School of Math. Sci.\\ Peking Univ.\\ Beijing 100871,
China\\ kcchang@math.pku.edu.cn}


\begin{document}

\maketitle

\begin{abstract} We develop a nonlinear spectral graph theory, in which the Laplace operator is replaced by the $1-$ Laplacian $\Delta_1$. The eigenvalue problem is to solve a nonlinear system involving a set valued function. In the study, we investigate the structure of the solutions, the minimax characterization of eigenvalues, the multiplicity theorem, etc. The eigenvalues as well as the eigenvectors are computed for several elementary graphs. The graphic feature of eigenvalues are also studied. In particular, Cheeger's constant, which has only some upper and lower bounds in linear spectral theory, equals to the first non-zero $\Delta_1$ eigenvalue for connected graphs.
\end{abstract}

\section{Introduction}

The Laplace operator is a differential operator acting on
functions defined on a manifold $M$, $\Delta u=-div (\nabla u)$.
It can be seen as the differential of the Dirichlet functional
$$D(u)=  \int_M |\nabla u(x)|^2 dx$$
on the Sobolev space $H^1(M)$. The eigenvalue problem is to find a pair $(\lambda, u)\in R^1 \times H^1(M)$
satisfying
$$ \Delta u=\lambda u.$$

While the $1-$ Laplace operator $\Delta_1$ is defined to be the subdifferential of the total variation functional $\|D u\|(M)$ for functions on Bounded Variation Space $BV(M)$, where
$$\|D u\|(M):=sup \{ \int_M u \,div (\varphi) dx\,|\, \varphi\in C^1_0(M, R^n), |\varphi_k(x)|\le 1,\, 1\le k\le n,\, a.e. x\in M\}.$$
is globally defined.
\vskip 0.5cm
Formally,
$$ \Delta_1 u=-div(\frac{\nabla \,u}{|\nabla \,u|}),$$
and the eigenvalue problem is to find a pair $(\mu, u)\in \mathbb{R}^1\times BV(M)$ satisfying
$$\Delta_1 u\in\mu Sgn(u).$$

As an example, Chang \cite{refC3} studied the eigenvalue problem for $\Delta_1$ on the unit interval $[0,1]$. The spectrum for the Neumann boundary problem is
$$ \sigma(\Delta_1)=\{2k\,|\, k=0, 1, 2,\cdots,\}$$
and the associate eigenfunctions are
$$ \pm \psi_k(x)=\pm sgn (cos(k\pi x)),$$
in contrast with the spectrum of the Neumann problem for $\Delta$ on the unit interval $[0,1]$:
$$ \sigma(\Delta)=\{2k\,|\, k=0, 1, 2,\cdots,\},$$
and the associate eigenfunctions:
$$ \pm \psi_k(x)=\pm cos(k\pi x).$$

The solutions for equations involving $\Delta$ are smooth, while those for $\Delta_1$ may be discontinuous. Since solutions in many interesting problems e.g., in the signal processing and in the image processing etc., may be discontinuous, the $1-$ Laplace operator $\Delta_1$ has been received much attention in recent years.
\vskip 0.5cm

Interestingly, Kawohl, B., Fridman, V. \cite{refKF} and Kawohl, B., Schuricht, F. \cite{refKS} studied the connection between Cheeger's constant $h(\Omega)$ (see Cheeger\cite{refCe}) and the first eigenvalue of $\Delta_1$ under Dirirchet boundary condition on a bounded domain $\Omega$ in $R^n$. They showed:
$$ \lim_{p\to 1+0}\lambda_p(\Omega)=h(\Omega),$$
where $\lambda_p(\Omega)$ is the first eigenvalue of the $p-$ Laplacian, and
$$\Delta_p u=div (|\nabla\, u|^{p-2}\nabla\, u), \,\,1<p<\infty $$
on $\Omega$ with Dirichlet boundary condition.

More recently, B\"{u}hler and Hein \cite{refBuH1}\cite{refBuH2} studied the $p-$ Laplacian on graphs, and found that the
Cheeger's constant $h(G)$ on a connected graph is the limit of the second eigenvalues for the $p-$ Laplacian, as $p\to 1$. Szlam and Bresson \cite{refSB} studied the relationship between the total variation of a graph and the Cheeger Cuts.

\vskip 0.5cm
$\Delta_1$ is in some sense the limit of $\Delta_p$ as $p\to 1,$ and is exactly the subdifferential of the total variation. This inspired us to study the eigenvalue problem for $\Delta_1$ on graphs, in particular, the connection between the first nonzero eigenvalue of $\Delta_1$ and Cheeger's constant.

\vskip 0.5cm
The Spectral theory for the Laplace operator on graphs is a fruitful field in graph theory, one can find many valuable monographs on this topics e.g., Brouwer, Haemers \cite{refBH}, Biyikpglu, Leydold and Stadler\cite{refBLS}, Chung\cite{refCh}  etc. and a vast of references therein.
\vskip 0.5cm
Given an undirected graph $G=(V, E)$ with vertex set $V=\{1, 2, \cdots, n\}$ and edge set $E$, each edge $e$ is a pair of vertices $(x, y)$. To the edge $e\in E$, we assign an orientation, let $x$  be the head, and $y$ be the tail, they are denoted by $x= e_h,$ and $ y= e_t$ respectively.
\vskip 0.5cm
Let $m$ be the number of edges in $E$.
An incidence matrix $B=(b_{ei})$ is a $m\times n$ matrix:
\begin{equation}b_{ei}=\left\{\begin{array} {l}
1 \,\,\,\,\,\,\,\,\mbox{if}\,\,\,i= e_h,\\
-1 \,\,\,\,\,\,\,\mbox{if}\,\,\,i= e_t.\\
0 \,\,\,\,\,\,\,\,i\notin e.
\end{array}\right.
\end{equation}
where $e\in E,\, i\in V.$

For any vertex $i$, $d_i$, the degree of $i$, is defined to be the number of all edges passing through $i$, i.e.,
$$ d_i= card(\{e\in E\,|\, i\in e\})\,\, i=1, 2,\cdots, n.$$
Let $D=diag\{ d_1, d_2,\cdots, d_n\}$, and $d=\Sigma^n_{i=1} d_i.$.

$B$ is the counterpart of the differential operator $grad$ on graphs, and the corresponding Laplace operator reads as:
$$ L=B^T B=(l_{ij}),$$
where
\begin{equation*}l_{ij}=\left\{\begin{array} {l}
-1 \,\,\,\,\,\,\,\,\mbox{if}\,\,\,{i,j}\in E, \mbox{and}\, i\neq j\\
d_i \,\,\,\,\,\,\,\mbox{if}\,\,\,i=j.\\
0 \,\,\,\,\,\,\,\,\mbox{otherwise}.
\end{array}\right.
\end{equation*}
It is easily seen that $L$ is independent to the choice of orientation.

The Chung's version \cite{refCh} of the Laplacian is modified to be:
\begin{equation}
\textbf{L}=D^{-1/2}L D^{-1/2},
\end{equation}
with the convention $D^{-1}_{ii}=0$ for $d_i=0$.

The eigenpair for $L$ on $G$ is the solution $(\lambda, \phi)\in R^1\times \mathbb{R}^n\backslash \{\theta\}$ of the system, where $\theta$ is the $0$ vector:
\begin{equation}
L \phi=\lambda D \phi.
\end{equation}

The Dirichlet function on a graph becomes
\begin{equation}
J(x)=\frac{1}{2}\Sigma_{i=1}^n\Sigma_{j\sim i}(x_i-x_j)^2,
\end{equation}
where $i\sim j$ means that $i$ is adjacent to $j$, and vice versa.

It is easy to see that an eigenvector $\phi$ of the system $(1.3)$ is a critical point of the Dirichlet function $J$ under the constraint:
\begin{equation*}
\Sigma^n_{i=1} d_i |x_i|^2=1.
\end{equation*}
While the eigenvalue $\lambda$ is the value of Dirichlet function at $\phi$.
\vskip 0.5cm

In a parallel way, we introduce the $1-$ Laplace operator on graphs, which precise formulation will be given in section 2.
\begin{equation}
\Delta_1 x= B^T Sgn(Bx),
\end{equation}
where $B$ is the incidence matrix, and $Sgn: \mathbb{R}^n \to (2^{\mathbb{R}})^n$ is a set valued mapping:
\begin{equation*}
Sgn(y)=(Sgn(y_1),Sgn(y_2),\cdots,  Sgn(y_n))\,\,\,\forall\, y=(y_1, y_2,\cdots, y_n),
\end{equation*}
in which
\begin{equation} Sgn(t)= \left\{\begin{array} {l}
1 \,\,\,\,\,\,\,\mbox{if}\,\,  t>0,\\
-1 \,\,\,\,\,\,\,\mbox{if} \,\, t<0,\\
$[-1, 1]$ \,\,\,\,\,\,\,\mbox{if}\,\,  t=0,
\end{array}\right.
\end{equation}
is a set valued function. The addition of two subsets $A, B\subset R^n$ is the set $\{x+y\,|\, x\in A, y\in B\},$ and for a scalar $\alpha$, the scalar multiplication $\alpha A$ is the set $\{\alpha x\,|\, x\in A\}.$

Also, $\Delta_1$ is independent to the choice of orientation.
\vskip 0.5cm

We use the notation:
\begin{equation}
 X=\{ x=(x_1, x_2, \cdots, x_n)\in \mathbb{R}^n\,|\,\Sigma^n_{i=1} d_i |x_i|=1.\}.
 \end{equation}

The eigenvalue problem (see Definition 2.4 below) for $\Delta_1$ is to find eigenpairs $(\mu, \phi)\in R^1\times X$ of the system, $\phi=(x_1, x_2, \cdots, x_n)$:
\begin{equation}\left\{\begin{array} {l}
\Sigma_{j\sim i} z_{ij}\in \mu d_i Sgn (x_i),\,\, i=1, \cdots, n,\\
z_{ij}\in Sgn(x_i-x_j), \\
z_{ji}=-z_{ij}.
\end{array}\right.
\end{equation}

In the following, $\mu$ is called an eigenvalue with eigenvector $\phi$. The set of all eigenvalues is called the spectrum of $\Delta_1.$

Note that the coefficients $z_{ij}=z_{ij}(x)$ depend upon the point $x$, they are called the adjacent coefficients.
\vskip 0.5cm

We expect that an eigenvector $\phi$ of the system $(1.8)$ is a critical point of the following energy function:
\begin{equation}
 I(x)=\Sigma^n_{i=1}\Sigma_{j\sim i} |x_i- x_j|,
\end{equation}
under the constraint: $x\in X,$ and the eigenvalue $\mu$ is the value of the energy function $I$ at $\phi$.

\vskip 0.5cm
In the system (1.8), the set valued function $Sgn(t)$ is involved. Not like in the linear spectral graph theory, which has rich mathematical tools at hand, in attacking the challengeable problem $(1.8)$, new methods are appealed.
\vskip 0.5cm

After testing few examples, we find that a huge set of solutions for the system (1.8) exists. In most cases, the solution sets appear in cells. The following are  natural questions:

1. How to approach these eigenpairs?

2. How to count the multiplicity of these solutions?

3. Do these eigenpairs reflect the feature of graphs?

4. What are the advantageous aspects of this theory?

\vskip 0.5cm

In this paper, a nodal domain decomposition technique in combining with methods in nonlinear analysis is developed in dealing with the nonlinear system (1.8).
\vskip 0.5cm
First, a nodal domain decomposition is introduced (Definition 3.2). By which the structure of solutions is studied. We prove all eigenvectors with respect to an eigenvalue, either being single points, or being cells with various dimensions (Theorem 3.7 and Corollary 3.8). As a result, the center of gravity can be seen as a representative of the cell, it is called a normalized solution. After multiplying a constant, the coordinates of normalized solutions are $1, 0,-1$. Moreover, the number of nodal domains of an eigenvector is estimated by the dimension of the cell.
\vskip 0.5cm
Second, as we mentioned earlier, in the linear theory, (1.3) is the variational equation for the Dirichlet function $J$ under the constraint $(1.5)$, while the system (1.8) is the variational equation for the function $I$ on $X.$ The nonsmooth setting of Liusternik-Schnirelmann theorem in the critical point theory is applied to the study of the multiplicity of solutions of $(1.8)$. Counting multiplicity (in the sense of genus), the number of eigenvalues of a graph $G$ with $n$ vertices is at least $n$ (see Theorem 4.13). This is the counterpart of the multiplicity result for the linear spectral theory.
\vskip 0.5cm

Third, many results in the linear spectral graph theory are extended to the nonlinear setting. e.g.,

1. The spectrum of $\Delta_1$ is in $ [0,1]$ (Theorem 2.8). Sufficient conditions on graphs with eigenvalues $\mu\in (0, 1)$ are given (Theorem 5.2).

2. A graph $G$ is connected if and only if the eigenvalue $0$ is of multiplicity 1. Moreover the multiplicity of eigenvalue $0$ equals to the  number of components of $G$ (Corollary 5.5).

3. Estimates for the first non-zero eigenvalue are studied in Theorem 5.13.

4. In the linear theory, the first non-zero eigenvalue is estimated via Rayleigh quotient, while in the nonlinear theory, a characterization of the mountain pass point plays a similar role (Theorem 5.12).

5. The spectrum as well as the normalized eigenvectors of some special graphs, e.g., paths, cycles, complete graphs, have been computed in section 6.

\vskip 0.5cm

Finally, it seems too early to say which are the advantageous aspects of the nonlinear theory. However, Theorem 5.15 reveals the fact that Cheeger's constant for a connected graph $G$ is exactly the second eigenvalue of $\Delta_1(G)$, but in the linear spectral theory, only some upper and lower bounds for Cheeger's constant can be given. The evidence sheds light on graphic feature of nonlinear eigenvalues.

\vskip 0.5cm

The paper is organized in six sections. The first section is an introduction. In the second section, the definition of $\Delta_1$ and its eigenvalue problem are introduced. A few elementary examples are provided, by which, readers may get a feeling on the nonlinear eigenvalue problem. Basic properties of the nonlinear spectrum are studied. The nodal domain decomposition for vectors on graphs and the structure of eigenvectors are obtained in the third section.. The fourth section is devoted to the critical point theory in the nonsmooth setting. The Liusternik Schnirelmann Theorem is extended. A crucial step in the application to our eigenvalue problem is Theorem 4.11. Basic results on eigenvalues are studied in Section 5, the characterization of the mountain pass point on graphs with applications are studied as well. The eigenvalues and eigenvectors for several elementary graphs are presented in Section 6.

\section{The $1-$ Laplacian on graphs }

\subsection{The definition}

A vector $x=(x_1, \cdots, x_n)\in \mathbb{R}^n$ can be seen as a function defined on $V$, $x: V\to \mathbb{R}^1$. The energy function associate to the incidence matrix $B$ is defined by
$$ I(x)= \sum_{e\in E} |(Bx)_e|=\sum_{e\in E} |\sum_{i=1}^n b_{ei} x_i|.$$

The subdifferential of the convex function $t\to |t|$ is the set valued function
$$\partial |t|=Sgn (t).$$

The following theorem holds:
\begin{theorem}$\forall\, x\in \mathbb{R}^n,$
\begin{equation}
u\in \partial I(x)\,\,\Leftrightarrow\,\, \exists\, z: E\to \mathbb{R}^1
\end{equation}
such that
\begin{enumerate}
\item $u=B^T z$,
\item $z_e(Bx)_e=|(Bx)_e|\,\,\,\,\,\forall\, e\in E.$
\end{enumerate}
\end{theorem}
\begin{proof}
$"\Rightarrow"$ $\forall\, e\in E,$ let
$$I_e(x)=|(Bx)_e|.$$
It is a convex function with subdifferential
$$\partial I_e(x)=\{\partial_k I_e(x)\,\,|\, 1\le k\le n \},$$
where
\begin{equation*}\partial_k I_e(x)= \left\{\begin{array} {l}
Sgn(Bx)_e \,\,\,\,\,\,\,\mbox{if}\,\, k= e_h,\\
-Sgn(Bx)_e \,\,\,\,\,\,\,\mbox{if}\,\, k= e_t,\\
0 \,\,\,\,\,\,\,\mbox{if}\,\, k\notin e.
\end{array}\right.
\end{equation*}
and
$$ \partial_i I_e(x)=-\partial_j I_e(x),\,\,\mbox{if}\,e=(i,j).$$
$\forall\, u\in \partial I(x)$, due to the additivity of the subdifferential of a convex function, there exist $u^e\in \partial I_e(x)$ such that
$$ u=\sum_{e\in E} u^e, $$
where $u^e=(u^e_1, u^e_2, \cdots, u^e_k)$ and
$$ u^e_k=\frac{\partial |x_i-x_j|}{\partial x_k}.$$
Therefore
$$ u^e_k=0\,\,\,\,\mbox{if}\,\, k\notin e.$$
If $e=(i,j)$, then
\begin{equation}
 u^e_i=-u^e_j,
\end{equation}
and
\begin{equation*}u^e_i= \left\{\begin{array} {l}
1\,\,\,\,\,\,\,\mbox{if}\,\,x_i>x_j,\\
-1 \,\,\,\,\,\,\,\mbox{if}\,\,x_i<x_j, \\
c \,\,\,\,\,\,\,\mbox{if}\,\, x_i=x_j,
\end{array}\right.
\end{equation*}
for some $c, |c|\le 1$.

Let us define
\begin{equation*}z_e= \left\{\begin{array} {l}
1\,\,\,\,\,\,\,\mbox{if}\,\,x_i>x_j,\\
-1 \,\,\,\,\,\,\,\mbox{if}\,\,x_i<x_j, \\
c \,\,\,\,\,\,\,\mbox{if}\,\,x_i=x_j.
\end{array}\right.
\end{equation*}
Then
$$ u^e_k=b_{ek} z_e,$$
and then
$$ u_k=\sum_{e\in E} u^e_k=\sum_{e\in E}b_{ek} z_e,$$
i.e.,
$$ u=B^T z.$$
This is $(1)$. Also from the definition
$$z_e(Bx)_e=|(Bx)_e|\,\,\,\,\forall\, e\in E.$$
This is $(2)$.

$"\Leftarrow"$ Conversely, from $(2)$, we have
$$ z_e\in Sgn (Bx)_e,$$
and then
$$ u^e=(B^Tz)_e \in \partial I_e(x).$$
From $(1)$, we obtain
$$ u=B^T z=\sum_{e\in E} u^e \in \partial I(x).$$
\end{proof}

The following Euler identity holds:
\begin{corollary} If $u\in \partial I(x)$, then
$ (u, x)=I(x).$
\end{corollary}
\begin{proof}
$$ (u, x)=(B^T z, x)=(z, B x)=\sum_{e\in E}(z_e, (Bx)_e)=\sum_{e\in E}|(Bx)_e|=I(x).$$
\end{proof}

\begin{definition} Let $G=(V, E).$ The set valued map
$$ \Delta_1: x\to \{B^T z \,|\,z: E\to \mathbb{R}^1 \mbox{is an}\, \mathbb{R}^m \,\mbox{vector, satisfying}\,  z_e(Bx)_e=|(Bx)_e|\,\forall\, e\in E\} $$
is called the $1-$ Laplacian on the graph $G$.
\end{definition}
It is rewritten as
$$\Delta_1 x=B^T  Sgn(Bx).$$
where $Sgn$ is defined in (1.7). Obviously, $\Delta_1: \mathbb{R}^n\to (2^{\mathbb{R}})^n$ is a nonlinear set valued mapping, which is independent to the special choice of orientation.

Thus, in computations, we shall always fix an orientation, and write
$$ -e=(j, i)\,\,\,\mbox{if}\,\,\,\,e=(i, j).$$
Following $(2.2),$
$$ z_e=-z_{-e},\,\,\,\,\,\,i.e., \,\,\,\,\,z_{ij}=-z_{ji}.$$
under this orientation we write the operator in the coordinate form:
\begin{align*}
&(\Delta_1 x)_i=(B^T Sgn(Bx))_i\\
&=\{\Sigma_{j\sim i} z_{ij}(x)\,|\,z_{ij}(x)\in Sgn(x_i-x_j),
z_{ji}(x)=-z_{ij}(x)\,\,\,\,\,\forall \,i\sim j\},\\
& i=1,2, \cdots, n.
\end{align*}

\begin{definition} $(\mu, x)\in \mathbb{R}^1\times X$ is called an eigen-pair of the $1-$ Laplacian  $\Delta$ on $G=(V, E)$, if
\begin{equation}
\mu D Sgn(x)\bigcap \Delta_1 x\neq \emptyset,
\end{equation}
where $D=diag\{d_1, \cdots, d_n\}$.

The set of all solutions of (2.3) is denoted by $S=S(G)$.
\end{definition}

In the coordinate form, $(2.3)$ is the system: $ z_{ij}(x)\in Sgn(x_i-x_j)$ satisfying $z_{ji}(x)=-z_{ij}(x),\,\forall\,i\sim j$, and
\begin{equation}
\Sigma_{j\sim i} z_{ij}(x)\in \mu d_i Sgn (x_i),\,\, i=1, \cdots, n,
\end{equation}

This is exactly (1.8). Following Corollary 2.2, we have
\begin{corollary} If $(\mu, x)$ is an eigenpair, then
$$ I(x)=\mu.$$
\end{corollary}
\begin{proof} Let $u= B^T z$,  we have
$$ I(x)=(u, x)=(B^T z, x)= \mu\Sigma^n_{i=1}d_i x_i sgn(x_i)=\mu\Sigma^n_{i=1}d_i |x_i|=\mu.$$
where
\begin{equation} sgn(t)= \left\{\begin{array} {l}
1 \,\,\,\,\,\,\,\mbox{if}\,\,  t>0,\\
-1 \,\,\,\,\,\,\,\mbox{if} \,\, t<0,\\
0 \,\,\,\,\,\,\,\mbox{if}\,\,  t=0,
\end{array}\right.
\end{equation}
\end{proof}

\subsection{Examples}
We present here few examples to illustrate the solutions of the above system. More examples will be given in section 6.
\vskip 0.5cm
$\mbox{Example 1.}$ $G=(V, E)$, where $V=\{1,2\}$, and $E=\{e=(1,2)\}$. The system reads as
\begin{equation*}\left\{\begin{array} {l}
 z_{12}\in \mu  Sgn (x_1),\\
 -z_{12}\in \mu Sgn (x_2),\\
z_{12}\in Sgn(x_1-x_2).
\end{array}\right.
\end{equation*}
Obviously we have two pairs of solutions:
\begin{equation*}\left\{\begin{array} {l}
\mu_1=0,\,\,(x_1, x_2)=\pm 1/2 (1,1),\,\,z_{12}=0,\\
\mu_2=1,\,\,(x_1, x_2)=\pm 1/2 (1,-1),\,\,z_{12}=1,.
\end{array}\right.
\end{equation*}
In fact, all solutions corresponding to $\mu_2$ are
$(x_1, x_2)=\pm (t, 1-t)\,\,\,\forall\, t\in[0,1].$

\vskip 0.5cm

$\mbox{Example 2.}$ $G=(V, E)$, where $V=\{1,2, 3\}$, and $E=\{e_1=(1,2), e_2=(2,3), e_3=(3,1)\}$. The system reads as
\begin{equation*}\left\{\begin{array} {l}
z_{12}-z_{31}\in 2\mu  Sgn (x_1),\\
z_{23}-z_{12}\in 2\mu  Sgn (x_2),\\
z_{31}-z_{23}\in 2\mu  Sgn (x_3),\\
z_{ij}\in Sgn(x_i-x_j),\,\,i, j=1,2,3.
\end{array}\right.
\end{equation*}
We find the following pairs of solutions:
\begin{equation*}\left\{\begin{array} {l}
\mu_1=0,\,\,\triangle_0=\pm \frac{1}{6}(1,1,1),\,\,z_{12}=z_{23}=z_{31}=0,\\
\mu_2=\mu_3=1,\,\,\pm l_1=\{\pm \frac{1}{2} (t,-1+t, 0,)\,\,\,t\in [0,1]\},\\
\pm l_2=\{\pm \frac{1}{2} (t,0, -1+t)\,\,\,t\in [0,1]\},\\
\pm l_3=\{\pm \frac{1}{2} (0,t, -1+t)\,\,\,t\in [0,1]\},\,\,z_{12}=1,z_{23}=-1,z_{31}=-1.
\end{array}\right.
\end{equation*}
The union of the last six segments consist of a cycle: $l_1\circ -l_3 \circ -l_2 \circ -l_1 \circ l_3 \circ l_2.$

\vskip 0.5cm
$\mbox{Example 3.}$ $G=(V, E)$, where $V=\{1,2, 3, 4\}$, and $E=\{e_1=(1,2),  e_2=(2,3),  e_3=(3,4)\}$. The system reads as
\begin{equation*}\left\{\begin{array} {l}
z_{12}\in \mu  Sgn (x_1),\\
z_{23}-z_{12}\in 2\mu Sgn (x_2),\\
z_{34}-z_{23}\in 2\mu Sgn (x_3),\\
-z_{34}\in \mu Sgn (x_4),\\
z_{ij}\in Sgn(x_i-x_j)\,\,i, j=1,2,3,4.
\end{array}\right.
\end{equation*}

We find the following pairs of solutions:
\begin{equation*}\left\{\begin{array} {l}
\mu_1=0,\,\,\triangle_0=\pm \frac{1}{6}(1,1,1,1),\,\,z_{12}=z_{23}=z_{34}=0,\\
\mu_2=\frac{1}{3},\,\,\pm \triangle^1_1=\{\pm (t, t, -(\frac{1}{3}-t), -(\frac{1}{3}-t))\,\,|\,t\in [0,\frac{1}{3}]\},\\
\mu_3=\mu_4=1,\,\,\triangle^2_1=\{\pm(t,0,0,(1-t))\,\,\,t\in [0,1]\},\,\,z_{12}=\frac{1}{3}, z_{23}=1, z_{34}=\frac{1}{3},\\
\triangle_2^1=\{\pm (t_1,-\frac{1}{2}(1-t_1-t_2),0,t_2)\,\,|\, t_1+t_2\le 1, t_1, t_2\ge 0\},\,\,z_{12}=1, z_{23}=1, z_{34}=1,\\
\triangle_2^2=\{\pm (t_1,0,-\frac{1}{2}(1-t_1-t_2),t_2)\,\,|\, t_1+t_2\le 1, t_1, t_2\ge 0\},\,\,z_{12}=1, z_{23}=-1, z_{34}=-1,\\
\triangle_3=\{ \pm(t_1,\frac{-t_2}{2}, \frac{t_3}{2}, -(1-t_1-t_2-t_3))\,\,|\, t_1+t_2+t_3\le 1, t_1, t_2, t_3\ge 0\}\\
z_{12}=1, z_{23}=-1, z_{34}=1.
\end{array}\right.
\end{equation*}

It is interesting to note:
\vskip 0.5cm
(1) All the solution sets are closed cells.
\vskip 0.5cm
(2) For the eigenvalue: $\mu=1$, i.e., $\mu_3$ and $\mu_4$, the associate critical set consists of a cycle. In fact, let $l_1=\Delta_1^2$, and $l_2=\{(-t, 0,0, (1-t))\,|\, t\in [0,1]\},$  then $l_1 \circ l_2\circ (-l_1) \circ (-l_2)$ is a cycle, while $l_2\subset \Delta_2^1\cap \Delta^2_2.$

\vskip 0.5cm

$\mbox{Example 4.}$ $G=(V, E)$, where $V=\{1,2, 3, 4\}$, and $E=\{e_1=(1,2), e_2=(2,3), e_3=(3,4), e_4=(4,1)\}$. The system reads as
\begin{equation*}\left\{\begin{array} {l}
z_{12}-z_{41}\in 2\mu  Sgn (x_1),\\
z_{23}-z_{12}\in 2\mu Sgn (x_2),\\
z_{34}-z_{23}\in 2\mu Sgn (x_3),\\
z_{41}-z_{34}\in 2\mu Sgn (x_4),\\
z_{ij}\in Sgn(x_i-x_j),\,\, z_{ij}=-z_{ji}\,\,i, j=1,2,3,4.
\end{array}\right.
\end{equation*}
We have four pairs of solutions:
\begin{equation*}\left\{\begin{array} {l}
\mu_1=0,\,\,(x_1, x_2, x_3, x_4)=\pm 1/8 (1,1,1,1),\\
\mu_2=\frac{1}{2},\,\,(x_1, x_2, x_3, x_4)=\pm 1/8 (1,1,-1,-1),\\
\mu_3=1,\,\,(x_1, x_2, x_3, x_4)=\pm 1/4 (1,0,-1,0),\\
\mu_4=1,\,\,(x_1, x_2, x_3, x_4)=\pm 1/8 (1,-1,1,-1).\\
\end{array}\right.
\end{equation*}

$\mbox{Example 5.}$ $G=(V, E)$, where $V=\{1,2, 3, 4, 5\}$, and $E=\{e_1=(1,4), e_2=(1,5), e_3=(4,5), e_4=(2,5), e_5=(3,5)\}$. The system reads as
\begin{equation*}\left\{\begin{array} {l}
z_{14}+z_{15}\in 2\mu  Sgn (x_1),\\
z_{25}\in \mu Sgn (x_2),\\
z_{35}\in \mu Sgn (x_3),\\
z_{41}+z_{45}\in 2\mu Sgn (x_4),\\
z_{51}+z_{52}+z_{53}+z_{54}\in 4\mu Sgn (x_5),\\
z_{ij}\in Sgn(x_i-x_j),\,z_{ij}=-z_{ji}\,\,i, j=1,2,3,4,5.
\end{array}\right.
\end{equation*}
We have five pairs of solutions:
\begin{equation*}\left\{\begin{array} {l}
\mu_1=0,\,\,(x_1, x_2, x_3, x_4, x_5)=\pm 1/10 (1,1,1,1,1),\\
\mu_2=\frac{1}{2},\,\,(x_1, x_2, x_3, x_4, x_5)=\pm 1/4 (1,0,0,1,0),\\
\mu_3=1,\,\,(x_1, x_2, x_3, x_4, x_5)=\pm 1/4 (1,-1,-1,0,0),\\
\mu_4=1,\,\,(x_1, x_2, x_3, x_4, x_5)=\pm 1/6 (1,1,-1,-1,0),\\
\mu_5=1,\,\,(x_1, x_2, x_3, x_4, x_5)=\pm 1/6 (1,-1,1,-1,0).
\end{array}\right.
\end{equation*}

\subsection{Basic facts}

We study the basic propositions of the eigenvalues for $1-$ Laplacian $\Delta_1$ on graphs. Let $G=(V, E)$ be a graph and let $(\mu, x)\in\mathbb{R}^1\times \mathbb{R}^n$ be an eigenpair for $\Delta_1(G)$

\begin{theorem}  If $\mu \neq 0,$ then $0\in \Sigma^n_{i=1} d_i Sgn(x_i).$
\end{theorem}
\begin{proof} Look at the system:
$$ \Sigma_{j\in i} z_{ij} \in \mu d_i Sgn (x_i),\,\,\,\forall\, i=1, \cdots, n.$$
Since
$$ \Sigma_{i=1}^n \Sigma_{j \sim i} z_{ij}=0,$$
we obtain
$$ 0\in \mu \sum^n_{i=1} d_i Sgn(x_i).$$
\end{proof}
\vskip 0.5cm

\textbf{Remark} Hein and Buehler \cite{refHB} showed similar results like Theorem 2.1, Corollary 2.2, Corollary 2.5 and Theorem 2.6 for unweighted $L^1$ norm 1-Laplacian.
In the graph theoretic terminology Theorem 2.6 is equivalent to say that $x$ has to have weighted median zero (weighed $L^1$-norm).

\vskip 0.5cm

\begin{theorem}  If $x\in X,$ with $x_1=x_2= \cdots x_n,$ then $x=\frac{1}{\sum^n_{i=1} d_i}(1, 1, \cdots, 1)$ is an eigenvector with eigenvalue $\mu=0.$
Conversely, if $G$ is connected, then $\mu=0$ is an eigenvalue with eigenvector $x=(x_1, \cdots x_n)$ where $x_1=x_2= \cdots x_n.$
\end{theorem}
\begin{proof}
$"\Rightarrow"$ First we prove: $x=\frac{1}{\sum^n_{i=1} d_i}(1, 1, \cdots, 1)$, is an eigenvector, i.e., it satisfies the system $(2.3)$ with $\mu=0,$
 or $0\in \Delta_1 x.$ In fact $I$ is a convex function on $\mathbb{R}^n$, it achieves its minimum at $\phi$, it follows
$$ 0\in \partial I(x).$$
Since
$$\partial_{x_i}I(x)=\sum_{j\sim i} Sgn (x_i-x_j),$$
This implies the existence of $z_{ij}\in Sgn (x_i-x_j)$ satisfying $(1.8)$ with $\mu=0.$

$"\Leftarrow"$ Since $G$ is connected, $\forall\, k, j\in V, \,\exists\,i_0, i_1,\cdots,i_l$ such that $k=i_0\sim i_1 \sim \cdots \sim i_l=j$. After Corollary 2.5, $I(x)=0$, it implies
$$ x_k=x_{i_1}=\cdots=x_{i_{l-1}} =x_j.$$
\end{proof}
\vskip 0.5cm

\begin{theorem} For any eigenvalue $\mu$ of $\Delta_1,$ we have
$$ 0\le \mu \le 1.$$
\end{theorem}
\begin{proof}Since
$$ |x_k-x_j|\le |x_k|+|x_j|,$$
we have
$$0\le I(x)=\Sigma_{j\sim k}|x_k-x_j|\le \Sigma_{j\sim k}(|x_k|+|x_j|)\le \Sigma_{j=1}^n
d_j |x_j|=1.$$
According to Corollary 2.5, $\mu=I(x)$, the lemma is proved.
\end{proof}

\section{The structure of the solution set}

For a given function, one considers the subsets of its domain, on which it is $0$ or $+$ or $-$. The subsets are called nodal domains.

First, an eigenvector is regarded as a function, we rewrite it by nodal domains.

Next we introduce the notion of cells. A cell $\Delta$ is a simplex $\{x=\Sigma^r_{k=1} \lambda_k \phi_k\,|\,\Sigma^r_{k=1}\lambda_k=1\, \lambda_k\ge 0, k=1, \cdots, r.  \}$, where $ \{\phi_1, \cdots, \phi_r\}$ are linearly independent vectors. The dimension of the cell is $r-1$. we shall prove that the set $S(G)$ of all eigenvectors consists of a set of cells.

Finally, we take the center of gravity of a cell in $S(G)$ as a normalized eigenvector.

\subsection{Nodal domains}

Given a graph $G=(V, E)$, for a vector $x=(x_1, \cdots, x_n)\in \mathbb{R}^n$, according to the signatures of $x_i$, we classify the vertices into three groups. Let
$$ D^0=\{i\in V\,|\, x_i=0 \}, D^{\pm}=\{i\in V\,|\, \pm x_i>0 \}.$$
We call $D^0$ the null set of $x$, and the vertex set of a connected component of the subgraph induced by $D^{\pm}$ is called a $\pm$ nodal domain. Accordingly, we divide $V$ into $r^+ + r^- +1$ disjoint $\pm$ nodal domains with the null set:
$$ V=\bigcup_{\alpha=1}^{r^+} D^+_\alpha \cup \bigcup_{\beta=1}^{r^-} D^-_{\beta} \cup D^0.$$
where $D^{\pm}_\gamma$ is a $\pm$ nodal domain, and $r^{\pm}$ is the number of $\pm$ nodal domains.

The adjacent relations connecting nodal domains and the nodal set can be summarized as the following equivalent statements:

\begin{enumerate}
\item $D^+_\alpha$ and $D^+_\beta$ have no connections $\forall\,\alpha\neq \beta$. The same is true if  $D^+$ is replaced by $D^-$,
\item If $j\sim i\in D^+_\alpha$ and $j\notin D^+_\alpha$, then $j\in \bigcup_{\beta=1}^{r^-}D^-_\beta \cup D^0$. The same is true if $i\in D^-_\beta$ and $j\notin D^-_\beta$, then $ j\in \bigcup_{\alpha=1}^{r^+}D^+_\alpha \cup D^0$,
\item If $i\in D^+_\alpha$, then the $i$th summation $\Sigma_{j\sim i} z_{ij}$ depends only on the connections inside $D^+_\alpha$ and those connections to $\bigcup_{\beta=1}^{r^-}D^-_\beta \cup D^0$.
\end{enumerate}

\vskip 0.5cm

\begin{definition} The number $r=r^+ +r^-$ is called the number of nodal domains of $x$.
\end{definition}

Let $\{\textbf{e}_1, \textbf{e}_2, \cdots, \textbf{e}_n\} $ be the Cartesian basis of $\mathbb{R}^n, \forall\, x\in X,$ let
$$ \textrm{E}^{\pm}_\gamma=\textrm{E}^{\pm}_\gamma(x)=\Sigma_{i\in D^{\pm}_\gamma} \textbf{e}_i,\,\,\,\delta^{\pm}_\gamma =\Sigma_{i\in D^{\pm}_\gamma} d_i,$$
$$\delta^{\pm}= \delta^{\pm}(x)=\Sigma \delta^{\pm}_\gamma(x),\,\,\,\,\delta^0=\delta^0(x)=d-\delta^+(x)-\delta^-(x).$$
\vskip 0.5cm
\begin{definition}
A vector $(x_1, x_2, \cdots, x_n)\in X$ is expressed according to its nodal domains as follow:
$$ x=(\Sigma^{r^+}_{\alpha=1}\Sigma_{i\in D^+_\alpha}-\Sigma^{r^-}_{\beta=1}\Sigma_{i\in D^-_\beta}) \xi_i \textbf{e}_i,$$
where $\xi_i=|x_i|,\,\forall\, i,$ and
$$ (\Sigma^{r^+}_{\alpha=1}\Sigma_{i\in D^+_\alpha}+\Sigma^{r^-}_{\beta=1}\Sigma_{i\in D^-_\beta}) \xi_i d_i=1.$$
It is called the nodal domain decomposition of $x$.
\end{definition}

Let $x\in S(G)$ be an eigenvector of $\Delta_1$.  According to Theorem 2.6, if $\mu\neq 0$, then
$$ 0\in (\Sigma^{r^+}_{\alpha=1}\Sigma_{i\in D^+_\alpha}+\Sigma^{r^-}_{\beta=1}\Sigma_{i\in D^-_\beta}+\Sigma_{i\in D^0}) m_i d_i,$$
where
\begin{equation*}m_i=m_i(x)=\left\{\begin{array} {l}
1,\,\,\,\,i\in \bigcup_{\alpha=1}^{r^+} D^+_\alpha,\\
-1,\,\,\,\,i\in \bigcup_{\beta=1}^{r^-} D^-_\beta,\\
c_i,\,\,\,\,\,i\in D^0.
\end{array}\right.
\end{equation*}
$|c_i|\le 1$, it follows
$$ |\delta^+-\delta^-|\le \delta^0.$$

The eigenpair system now is rewritten as
\begin{equation}
 \Sigma_{j\sim i}z_{ij}(x)= \mu m_i(x)d_i,\,\,\,\,\,\,\forall\,i.
\end{equation}

Let
$$Z^+_{\alpha\beta} =\Sigma_{(i,j)\in D^+_\alpha \times D^-_\beta} z_{ij},\,\, Z^-_{\beta\alpha} =\Sigma_{(i,j)\in D^-_\beta\times D^+_\alpha} z_{ij},$$
denote the numbers of edges between $D^+_\alpha$ and $D^-_\beta$, and let
$$ Z^\downarrow_\alpha=\Sigma_{(i,j)\in D^+_\alpha \times D^0} z_{ij},\,\,Z^\uparrow_\beta=\Sigma_{(i,j)\in D^-_\beta\times D^0} z_{ij},$$
denote the numbers of edges between $D^+_\alpha$ ($D^-_\beta$) and $D^0$ respectively. Let

$$Z^\uparrow_\alpha=\Sigma_{(i,j)\in D^0 \times D^+_\alpha } z_{ij},\,\,Z^\downarrow_\beta=\Sigma_{(i,j)\in D^0 \times D^-_\beta} z_{ij}.$$
From $z_{ij}=-z_{ji}$,  it follows
$$ Z^+_{\alpha\beta}=-Z^-_{\beta\alpha},$$
and
$$ Z^\downarrow_\alpha=-Z^\uparrow_\alpha,\,\,Z^\uparrow_\beta=-Z^\downarrow_\beta.$$

For any $D^{\pm}_\tau, $ let
\begin{equation}
 p^{\pm}_i= \Sigma_{j\sim i, j\notin D^{\pm}_\tau} z_{ij},\,\,\,\forall\, i\in D^{\pm}_\tau.
\end{equation}
They denote the numbers of edges between $i\in D^{\pm}_\tau$ and those vertices outside $D^{\pm}_\tau$.

The following equations relate vertices in different nodal domains:

\begin{align}
& \Sigma_{i\in D^+_\alpha} p^+_i= \Sigma_\beta Z^+_{\alpha\beta}+Z^\downarrow_\alpha.\\
& \Sigma_{i\in D^-_\beta} p^-_i= \Sigma_\alpha Z^-_{\beta\alpha}+Z^\uparrow_\beta.
\end{align}
By adding the following equations over $D^{\pm}_\tau$
$$ \Sigma_{j\sim i} z_{ij}= \pm \mu d_i,\,\,\forall i\in D^{\pm}_\tau,$$
we obtain:
\begin{equation}
\Sigma_{i\in D^{\pm}_\tau} p^{\pm}_i= \pm \mu \Sigma_{i\in D^{\pm}_\tau} d_i =\pm \mu \delta^{\pm}_\tau,
\end{equation}
provided by
\begin{equation}
\Sigma_{i \in D^{\pm}_\tau}\Sigma_{j\sim i, j\in D^{\pm}_\tau} z_{ij}= 0.
\end{equation}

Combining (3.3), (3.4) and (3.5), we have
\begin{align}
& \Sigma_\alpha(\Sigma_\beta Z^+_{\alpha\beta} +Z^\downarrow_\alpha)= \mu \delta^+,\\
& \Sigma_\beta(\Sigma_\alpha Z^-_{\beta\alpha} +Z^\uparrow_\beta)= -\mu \delta^-.
\end{align}
\vskip 0.5cm

\begin{lemma}
$$ \Sigma_{i\in D^{\pm}_\tau}\Sigma_{j\in D^{\pm}_\tau,\,j\sim i} |x_i-x_j|+\Sigma_{i\in D^{\pm}_\tau } p^{\pm}_i x_i=\mu\Sigma_{i\in D^{\pm}_\tau } d_i x_i,\,\,\,\,\forall\,\tau.$$
\end{lemma}
\begin{proof} Multiplying the following system
$$ \Sigma_{j\sim i} z_{ij} = \mu d_i,\,\,\,\,i=1, \cdots, n$$
by $x_i$, and adding them over $D^{\pm}_\tau,$ we obtain
$$ \Sigma_{i\in D^{\pm}_\tau } (\Sigma_{j\sim i, j\in D^{\pm}_\tau} +  \Sigma_{j\sim i, j\notin D^{\pm}_\tau}) z_{ij} x_i = \mu \Sigma_{i\in D^{\pm}_\tau} d_i x_i.$$
Following (3.2) and the relation:
$$ \Sigma_{i \in D^{\pm}_\tau} \Sigma_{j\sim i,\, j\in D^{\pm}_\tau} z_{ij} x_i= \Sigma_{i \in D^{\pm}_\tau}\Sigma_{j\sim i,\,j\in D^{\pm}_\tau} |x_i-x_j|,$$
the equation follows.
\end{proof}

\vskip 0.5cm
Similarly we have
\begin{lemma}
\begin{align}
& \Sigma_\alpha Z^\downarrow_\alpha= \Sigma_\beta Z^\uparrow_\beta.\\
& \Sigma_\alpha Z^\uparrow_\alpha= \Sigma_\beta Z^\downarrow_\beta.
\end{align}
\end{lemma}
\begin{proof}
By the relationship between $Z^\uparrow$ and $Z^\downarrow$, these two equations are the same. We only need to prove the first one. $\forall\, i\in D^0$ from $z_{ij}=-z_{ji}$, one sees
$$\Sigma_{i\in D^0} \Sigma_{j\sim i, \,\,j\in D^0} z_{ij} =0,$$
and
$$ \Sigma_{i=1}^n\Sigma_{j\sim i} z_{ij}=0.$$
Also, we have
\begin{equation*}
\Sigma_{i=1}^n\Sigma_{j\sim i} z_{ij}= \\
(\Sigma_{i\in D^0}\Sigma_{j\sim i, \,\,j\in D^0}+ \Sigma_\alpha\Sigma_{j\sim i, \,\,j\in D^+_\alpha}+ \Sigma_\beta \Sigma_{j\sim i, \,\,i,j\in D^-_\beta})z_{ij}=\\
\Sigma_\alpha Z^\uparrow_\alpha- \Sigma_\beta Z^\downarrow_\beta.
\end{equation*}
Combining these three equations together, the equation is proved.
\end{proof}

\subsection{Homotopic equivalence}

Given a subset $A\subset X.$ Let $x, y\in A$, we say that $x$ is equivalent to $y$ in $A$, denoted by $x\simeq y$, if there is a path $l$ connecting $x$ and $y$ in $A$, i.e.,$ \exists$ a continuous $ l: [0, 1]\to A$ such that $l(0)=x, l(1)=y.$
\vskip 0.5cm
\begin{theorem}
Let $x=(x_1, \cdots, x_n)\in S(G)$, and
$$\bar{\xi}= (\delta^{\pm}_\alpha)^{-1}\Sigma_{i\in D^{\pm}_\alpha} d_i x_i, \,\,\xi=\bar{\xi}\textrm{E}^{\pm}_\alpha+\Sigma_{j\notin D^{\pm}_\alpha} x_j \textbf{e}_j.$$
Then $\xi\simeq x$ in $S(G)\cap I^{-1}(c)$, where $c=I(x).$
\end{theorem}
\begin{proof}
We only prove the case $i \in D^+_\alpha $, the rest is the same. We define
$$ x^t=(1-t)x +t \xi,$$
and use the notation: $ \|x\|=\Sigma^n_{j=1} d_j|x_j|.$
We verify:

$1^o. \,\, x^t\in X, i.e., \|x^t\|=1.$ In fact,
$$ \|x^t\|=\Sigma_{j\notin D^+_\alpha} d_j |x_j|+\delta^+_\alpha \bar{\xi}=(\Sigma_{j\notin D^+_\alpha}+\Sigma_{j\in D^+_\alpha})d_j |x_j| =1.$$

$2^o. \,\, x^t\in S(G).$ It is sufficient to verify the system:
\begin{equation}
\Sigma_{j\sim i} z^t_{ij} \in \mu d_i Sgn (x^t_i),\,\,\,\forall\, i.
\end{equation}
where $z^t_{ij}\in Sgn(x^t_i-x^t_j).$ In fact, if at least one of $i$ and $j$ are not in $D^+_\alpha,$ then we have $z^t_{ij}=z_{ij}$, because the order between $x^t_i$ and $x^t_j$ is not changed. Therefore the LHS does not change.

In case both $i$ and $j$ are in $D^+_\alpha,$ when $0\le t<1$, the order between $x^t_i$ and $x^t_j$ is again not changed.  For $t=1$ we can just keep the signs of the differences of the original $x_i$ in order to satisfy the system $(3.11)$. Thus $x^1\in S(G),$ and then $x^t\in S(G)\, \,\,\forall\, t\in [0,1].$

$3^o. \,\, I(x^t)=I(x)\,\,\forall\,\,t\in [0,1].$ We split the sum into three parts (the adjacent relation $\sim$ is symmetric):
\begin{align*}
I(x^t)&=\Sigma_{i=1}^n\Sigma_{j \sim i} |x_i^t-x_j^t|\\
&=(\Sigma_{i\notin D^+_\alpha}\Sigma_{j \sim i,\,j\notin D^+_\alpha}+\Sigma_{i\in D^+_\alpha}\Sigma_{j \sim i,\, j\notin D^+_\alpha} +
\Sigma_{i\in D^+_\alpha}\Sigma_{j \sim i,\,j\in D^+_\alpha})|x^t_i-x^t_j|\\
&=I+ II+ III.
\end{align*}
Since $x^t_k=x_k,\,\forall\, k\notin D^+_\alpha, I$ is invariant. Moreover,
$$ III= \Sigma_{i\in D^+_\alpha}\Sigma_{j \sim i,\,j\in D^+_\alpha}|x^t_i-x^t_j|=(1-t)\Sigma_{i\in D^+_\alpha}\Sigma_{j \sim i,\,j\in D^+_\alpha}|x_i-x_j|,$$
and since $\forall\, i\in D^+_\alpha,\,j\notin D^+_\alpha,\,j\sim i,$ implies $x_i>x_j$,
\begin{align*}
II&=\Sigma_{i\in D^+_\alpha}\Sigma_{j \sim i,\, j\notin D^+_\alpha} (x^t_i-x^t_j) \\
&=\Sigma_{i\in D^+_\alpha}\Sigma_{j \sim i,\,j\notin D^+_\alpha}[(x_i-x_j)+t(\bar{\xi}-x_i)].\\
\end{align*}
In this case, $z_{ij}=1$. From lemma 3.3 and $(3.5)$, we have
\begin{align*}
&\bar{\xi}\Sigma_{i\in D^+_\alpha} p^+_i
=\mu \delta^+_\alpha \bar{\xi}\\
&=\mu\Sigma_{i\in D^+_\alpha}d_i x_i\\
&=\Sigma_{i\in D^+_\alpha}\Sigma_{j \sim i,\,j\in D^+_\alpha}|x_i-x_j|+ \Sigma_{i\in D^+_\alpha} p^+_i x_i,
\end{align*}
we arrive at
$$ II +III= (\Sigma_{i\in D^+_\alpha}\Sigma_{j \sim i,\,j\in D^+_\alpha}+ \Sigma_{i\in D^+_\alpha}\Sigma_{j \sim i,\, j\notin D^+_\alpha})|x_i-x_j|.$$
This prove $I(x^t)=I(x),\,\, t\in [0, 1]$.
\end{proof}
\vskip 0.5cm
By the theorem, any eigenvector $x$ is equivalent to an eigenvector of the form
\begin{equation} \left\{\begin{array} {l}
y=\Sigma_\alpha y^+_\alpha \textrm{E}^+_\alpha-\Sigma_\beta y^-_\beta \textrm{E}^-_\beta,\,\,y^+_\alpha, y^-_\beta>0,\\
\Sigma_\alpha y^+_\alpha \delta^+_\alpha + \Sigma_\beta y^-_\beta \delta^-_\beta=1.
\end{array}\right.
\end{equation}
i.e., on each nodal domain, the components of $y$ are constants. In this case, the function $I$ can be expressed as
\begin{theorem}
$$ I(y)=\Sigma_{\alpha,\beta} Z^+_{\alpha\beta}(y^+_\alpha +y^-_\beta) +\Sigma_\alpha Z^\downarrow_\alpha y^+_\alpha +\Sigma_\beta Z^\downarrow_\beta y^-_\beta.$$
\end{theorem}
\begin{proof}
Since there is no connections between any two positive nodal domains: $D^+_{\alpha_1}$ and $D^+_{\alpha_2},\,\, \alpha_1 \neq \alpha_2$, neither for any two negative nodal domains,
\begin{align*}
RHS&=\Sigma_{\alpha,\beta} \Sigma_{i\in D^+_\alpha}\Sigma_{j\sim i, j\in D^-_\beta}  z_{ij}(y)(y_i-y_j)\\
&+\Sigma_\alpha\Sigma_{i\in D^+_\alpha} \Sigma_{j\sim i, j\in D^0}  z_{ij}(y)y_i
+\Sigma_\beta \Sigma_{i\in D^0}\Sigma_{j\sim i, j\in D^-_\beta}  z_{ij}(y)(-y_j)\\
&=\Sigma_{i=1}^n\Sigma_{j\sim i}z_{ij}(y)(y_i-y_j)\\
&=\Sigma_{i=1}^n\Sigma_{j\sim i} |y_i-y_j| =I(y).
\end{align*}
\end{proof}

\subsection{Cell structure for eigenvectors}

Keeping those equivalent statements of the adjacent relations connecting nodal domains and the nodal set in mind and combining them with Theorem 3.5, we obtain
\begin{theorem} Let $\phi\in S(G)$ be of the form $(3.12)$. Then vectors in the cell
\begin{equation*}\begin{array} {l}
\Delta_{r^+ +r^- -1}=\{\psi=\Sigma_{\alpha=1}^{r^+} y^+_\alpha \textrm{E}^+_\alpha-\Sigma_{\beta=1}^{r^-} y^-_\beta \textrm{E}^-_\beta,\,|\,y^+_\alpha, y^-_\beta\ge 0,\,\,
\Sigma_{\alpha=1}^{r^+} y^+_\alpha \delta^+_\alpha + \Sigma_{\beta=1}^{r^-} y^-_\beta \delta^-_\beta=1\}
\end{array}
\end{equation*}
are all eigenvectors with the same eigenvalue as $\phi$.
\end{theorem}
\begin{proof} Since $\phi$ is an eigenvector, there exist $\mu\in [0,1],\, z_{ij}(\phi)\in Sgn(x_i-x_j),\, m_i(\phi)\in Sgn(x_i),$ such that
$$ \Sigma_{j \sim i} z_{ij}(\phi)= \mu d_i m_i(\phi).$$
However for any vector $\psi=(y_1, y_2, \cdots, y_n)\in \Delta_{r^+ +r^- -1}$, the signatures of $\psi$ in each nodal domain as well as in $D^0$ are the same as those of $\phi$. We choose $m_i(\psi)= m_i(\phi)$.  While in each nodal domain  $y_i,\, i=1,2,\cdots, n$, are constants, all the coordinates at adjacent vertices do not change their order. Therefore all terms $Sgn (x_i-x_j)\subset Sgn (y_i-y_j)\, \forall j\sim i$, and then $m_i(\psi)=m_i(\phi)= \Sigma_{j \sim i}z_{ij}(\phi)\in \Sigma_{j \sim i}Sgn(y_i-y_j)$.

Thus, $\psi$ satisfies the same system, i.e., it is an eigenvector with the same eigenvalue $\mu$.
\end{proof}
\vskip 0.5cm
\begin{corollary} If $x=(x_1, \cdots, x_n)$ is an eigenvector with $r$ nodal domains, then
$x\in \Delta_{r-1}\subset S(G)$, i.e., $x$ lies on a $r-1$ cell, which consists of eigenvectors.
\end{corollary}

\subsection{Normalization}

As a corollary of Theorem 3.5 and Corollary 3.8, we can make the components of the eigenvector being constant on all nodal domains. Namely
\begin{theorem} Let $x=(x_1, \cdots, x_n)$ be an eigenvector of the $1-$ Laplacian $\Delta_1(G)$ on a graph $G=(V, H)$. Then $x\simeq \hat{x}$, where
$$ \hat{x}=\delta^{-1}[\Sigma_\alpha \textrm{E}^+_\alpha- \Sigma_\beta\textrm{E}^-_\beta],\,\,\delta=\delta^+ +\delta^-,$$
and $E^\pm_\gamma$ are the nodal domains with respect to $x$.
\end{theorem}

Now, we arrive at the following
\begin{definition} An eigenvector $x$ of the $1-$ Laplacian $\Delta_1(G)$ on a graph $G=(V, E)$ is called normal, if it is of the form
\begin{equation}
 x=\delta^{-1}[\Sigma_\alpha \textrm{E}^+_\alpha- \Sigma_\beta\textrm{E}^-_\beta],
\end{equation}
where
$$ \textrm{E}^{\pm}_\tau=\Sigma_{i\in D^{\pm}_\tau} \textbf{e}_i,\,\,\delta^{\pm}_\tau=\Sigma_{i\in D^{\pm}_\tau } d_i.$$
$$ \delta^+=\Sigma^{r^+}_{\tau=1}\delta^+_\tau, \,\, \delta^-=\Sigma^{r^-}_{\tau=1}\delta^-_\tau,\,\,\,\delta=\delta^+ +\delta^-.$$
and $D^{\pm}_\tau, \,\tau=1,2,\cdots, r^{\pm}$, are all nodal domains with respect to $x.$
\end{definition}
\vskip 0.5cm
Our theorem 3.9 is restated as
\begin{theorem} Any eigenvector $x$ of $\Delta_1(G)$ is equivalent to a normal eigenvector within the set of eigenvectors with the same eigenvalue $\mu.$
\end{theorem}

Let us introduce a subset of $X$ as follow:
$$ \pi = \{x=(x_1, \cdots, x_n)\in X\,||\delta^+(x)-\delta^-(x)|\le \delta^0(x)\}.$$
The following statement is deduced from Theorem 3.9 and Theorem 2.6 directly:
\begin{theorem} Any eigenvector $x$ of $\Delta_1(G)$ with eigenvalue $\mu \neq 0$, lies on $\pi$.
\end{theorem}

\section{Multiplicity and critical point theory on piecewise linear manifolds}

In this section we study the critical point theory of the function $I$ (see (1.9)) on the set $X$ (see (1.7)). The purpose of the study is twofold:

(1) Provide a variational formulation of the eigenvalue problem for $\Delta_1(G)$, which has an explanation on the motivation of definition 2.4 on the eigenvalues of graphs.

(2) Define the multiplicity of eigenvalues. In the linear spectral theory, as a simple application of linear algebra, a graph with $n$ vertices possesses $n$ linearly independent eigenvectors, see for instance, Biyikpglu, Leydold and Stadler\cite{refBLS}, Brouwer and Haemers \cite{refBH}, and Chung\cite{refCh}. However, in the previous section, we have shown that the eigenvectors for the nonlinear operator $\Delta_1(G)$ appear in the form of cells, which are infinite subsets, except the $0-$ cells. How do we measure the multiplicity of these eigenvectors? Do we have some sort of similar multiplicity result in this case? A natural idea in mind is the Liusternik-Schnirelmann theory in nonlinear eigenvalue problems on symmetric differential manifolds.
\vskip 0.5cm

However,  the function $I$ and the manifold $X$ are not smooth. The extension of Liusternik-Schnirelmann theory to nonsmooth setting can be found in Chang\cite{refC1} and Corvellec, Degiovanni and Marzucchi\cite{refCDM} etc. However, all these extensions are abstract, in the application to our problem it is required to concretize the abstract theory. Subsections 4.1-4.2 are devoted to study the tangent space structure. Subsection 4.3 studies the projections on the tangent space.  Theorem 4.2 is crucial in characterizing the critical set of the function $I$. We carefully write down the pseudo-subgradient vector field of $I$ on $X$. The purpose is twofold: (1) Bridging up the relationship between the critical set of $I$ on $X$ and the set of all eigenvectors $S(G)$, and (2) it might be useful in computing the eigenvalues and eigenvectors numerically.

The non-smooth function $I$ on the piecewise linear manifold $X=\{x=(x_1, \cdots, x_n)\in\mathbb{R}^n\,|\, \sum^n_{i=1} d_i|x_i|=1 \}$ are taken into consideration, where $d_1, \cdots, d_n$ are the degrees of vertices. We shall follow the following steps:

1. Clarifying the notion of the critical point of the function $I$ on $X$.

2. Extending the Liusterink- Schnirelmann theory to $I$ on $X$. see Theorem 4.9

3. Building up the connection between the critical set of $I$ on $X$ and the set of all eigenvectors $S(G)$. It is Theorem 4.11.

4. Defining the multiplicity of an eigenvalue, see Definition 4.12.

\subsection{Decomposition}
\vskip 0.5cm
First, we decompose the piecewise linear manifold into pieces of open linear manifolds with different dimensions in addition to several isolated points.

A $k$ index subset is defined to be $\iota =\{i_1, \cdots, i_k\}$, $1\le k\le n$ with $1\le i_1<i_2 \cdots < i_k\le n.$ Let $I_k$ be the set of all $k$ indices. The number of $I_k$ is $|I_k|=C^n_k$, where $C^n_k$ are the binomial coefficients.

Define a vector valued mapping $m: I_k \to 2^k$, i.e.,
$$ m(\iota)=(m(\iota)_{i_1}, \cdots, m(\iota)_{i_k}),$$
where $m(\iota)_i\in \{+1, -1\},\,\forall \,\,i\in \iota.$
The set of all these mappings is denoted by $M_k$. The total number of $M_k, |M_k|=C^n_k 2^k.$

\vskip 0.5cm
We introduce the following notations:
$\forall\, (\iota, m)\in I_k \times M _k$,
\begin{equation}
\Gamma_{\iota, m}=\{ x=\Sigma_{\alpha\in \iota} x_\alpha m_\alpha\textbf{e}_{\alpha}\,\,|\,\,\Sigma_{\alpha\in \iota} x_\alpha d_{\alpha}=1,\,\, x_\alpha>0, \,\forall\,\,\alpha\in \iota\},
\end{equation}
where $d_1, \cdots, d_n$ are the degrees of vertices, and $m_\alpha=m(\iota)_{i_\alpha}$.

It is a $k-1$ cell, $k=1, 2, \cdots, n$. All these cells are open, except $k=1$.

Let $S_{k-1}$ be the set of $2^k C^n_k$ disconnected $k-1$ cells, i.e.,
$$ S_{k-1}=\bigcup_{(\iota, m)\in I_k\times M_k}\Gamma_{\iota, m}. $$
\vskip 0.5cm

Obviously we have the following propositions:

\begin{enumerate}
\item $ \Gamma_{\iota, m}\bigcap \Gamma_{\iota', m'}=\emptyset,\,\,\mbox{if}\,\,  (\iota, m)\neq (\iota', m'),$
\item $ S_i\bigcap S_j=\emptyset,\,\,\mbox{if}\,\,i\neq j,$
\item $ X=\bigcup_{i=0}^{n-1} S_i.$
\end{enumerate}

In particular, $\forall\, a\in X,$ there must be unique $k$ such that $a\in S_{k-1},$ and then  $\exists\,\,(\iota, m)\in I_k\times M_k$, such that $a\in \Gamma_{\iota, m}$, i.e.,
$$ a=\Sigma_{\alpha\in \iota} a_\alpha m_\alpha \textbf{e}_{\alpha}, \, \Sigma^k_{\alpha=1} a_\alpha d_{i_\alpha}=1,\,\,a_\alpha>0, \forall\,\alpha\in \iota.$$

Given two cells $\Gamma_{\iota, m}$ and $\Gamma_{\iota', m'}$, if
$$ \iota \subset \iota', \,\,\,\mbox{and}\,\,\,m'(\iota')|_\iota=m(\iota),$$
then
$$ \Gamma_{\iota, m}\subset \overline{\Gamma_{\iota', m'}}.$$
In this case we write $\Gamma_{\iota', m'}  \succeq \Gamma_{\iota, m}$ and say: the level of $\Gamma_{\iota', m'} $ is higher than that of $\Gamma_{\iota, m}$. In order to specify the relation, we sometimes write $\iota'=\iota\oplus \sigma,\,m'=m\oplus w,$ and  if there is no confusion, we also briefly write $\Gamma_{\iota', m'}$ as $\Delta_{\sigma, w}$.
\vskip 0.5cm
There are totally $3^{n-k}-1$ cells over $\Gamma_{\iota, m}$. Among them the numbers of
$(k+l-1)$ cells is $C^{n-k}_l 2^l,\,\,l=1,2,\cdots, n-k$.

Let $F^{k+l}=\{\Gamma_{\iota', m'}=\Delta_{\sigma, w},|\,\, \iota'=\iota\oplus \sigma\,\,m'=m\oplus w\},
l=1,2,\cdots, n-k$, and $F^k=\Gamma_{\iota, m}$. We have
$|F^{k+l}|=C^{n-k}_l 2^l$, and let
\begin{equation}
\textbf{F}_a=\bigcup_{l=0}^{n-k} F^{k+l}.
\end{equation}
 Given a point $a\in \Gamma_{\iota, m}\in S_{k-1}$, its neighborhood $B_\delta(a)$ on $X$ for small enough $\delta$ lies on $3^{n-k}$ pieces of various dimensional cells:
$$ B_\delta(a)= \bigsqcup_{\Delta_{\sigma, w}\in \textbf{F}_a}  B_{\sigma, w}(a), $$
where $\sigma=(j_1, \cdots, j_l), \,\,  w=(w_{j_1}, \cdots, w_{j_l}),$ and let $m_\alpha=m(\iota)|_{i_\alpha}, \, w_\beta=w(\sigma)|_{j_\beta}$,

\begin{align*}
B_{\sigma, w}(a)=&\{x\in B_\delta(a)\,|\, x=\Sigma_{\alpha\in \iota} \xi_\alpha m_\alpha \textbf{e}_{\alpha} +\Sigma_{\beta\in \sigma}\eta_\beta w_\beta \textbf{e}_{\beta}\\
&\xi_\alpha\in R^1, \eta_\beta>0,\,\,\Sigma_\alpha \xi_\alpha d_{\alpha}+\Sigma_{\beta\in \sigma} \eta_\beta d_{\beta}=1 \}.
\end{align*}

Obviously, $a$ lies in the interior of $\Gamma_{\iota, m}$, and is on the boundary of all those cells $\Delta_{\sigma,w}\in \textbf{F}_a\backslash \Gamma_{\iota, m} $.

\subsection{Tangent space}

Since now $X$ is a piecewise linear manifold, we shall extend the notion of tangent space. $h\in \mathbb{R}^n$ is called a tangent vector of $X$ at $a\in X$, if $a+th\in X$ for small $t>0.$ All tangent vectors form a tangent cone of $X$ at $a$, denoted by $T_a(X)$.

For any $\alpha, \gamma \in \iota$ with $\alpha\neq \gamma$, we define
$$ h_{\alpha,\gamma}=\frac{m_\alpha\textbf{e}_{\alpha}}{d_{\alpha}}-\frac{m_\gamma \textbf{e}_{\gamma}}{d_{\gamma}},$$
and $\forall\,\beta\in \sigma,$
$$ \psi^w_\beta=\frac{w_\beta\textbf{e}_{\beta}}{d_{\beta}}-\frac{m_\gamma \textbf{e}_{\gamma}}{d_{\gamma}}.$$

We have
\begin{theorem} If $a\in \Gamma_{\iota, m}\subset S_{k-1},$ then
\begin{equation*}
T_a(X)= \bigcup_{\Delta_{\sigma, w}\in \textbf{F}_a}T_a(B_{\sigma, w}(a)),
\end{equation*}
where
\begin{equation}
 T_a(B_{\sigma, w}(a))=\{ \Sigma_{\alpha\in (\iota\backslash\{\gamma\})} c_\alpha h_{\alpha,\gamma}+\Sigma_{\beta\in \sigma} b_\beta \psi^w_\beta\,|\, c_\alpha\in R^1, \,b_\beta> 0 \}.
\end{equation}
\end{theorem}
\begin{proof} First, for $\varepsilon>0$ small, we have
$$ a+\varepsilon h_{\alpha,\gamma}\in \Gamma_{\iota, m}.$$
Thus,
$$ h_{\alpha,\gamma}\in T_a(B_{\sigma, w}(a)).$$
Obviously the rank of these vectors $\{h_{\alpha,\gamma}\,|\,\alpha, \gamma\in \iota\}$ is $k-1$. They form the first group in the bracket of $(4.3)$.
\vskip 0.5cm

It is worth to note:
$$ h_{\alpha,\gamma} = - h_{\gamma, \alpha}.$$

Next, $\forall\, \beta\in \sigma$ and $\forall\, w: \sigma\to \{+1, -1\}$, we define
\begin{equation}
 \xi_{\beta}= (-a+\frac{w_\beta\textbf{e}_{\beta}}{ d_{\beta}}),
\end{equation}
where $w_\beta=w(\sigma)|_{\beta}$.
it follows
$$ a+\varepsilon \xi_{\beta}\in B_{\sigma, w}(a)$$
for $\varepsilon>0$ small, that is
$$ \xi_{\beta}\in T_a(B_{\sigma, w}(a)).$$
\vskip 0.5cm

Note
$$ a=\Sigma_{\alpha\in (\iota\backslash\{\gamma\})  } a_\alpha d_{\alpha} h_{\alpha, \gamma}+\frac{m_\gamma \textbf{e}_{\gamma}}{d_{\gamma}}, $$
we have
$$\psi^w_\beta=
\xi_{\beta}+\Sigma_{\alpha\in (\iota\backslash\{\gamma\}) }a_\alpha d_{\alpha} h_{\alpha, \gamma}.$$

Therefore
$$ \Sigma_{\alpha\in (\iota\backslash\{\gamma\})  } c_\alpha h_{\alpha, \gamma}+\Sigma_{\beta\in \sigma} b_\beta \psi^w_\beta= \Sigma_{\alpha\in (\iota\backslash\{\gamma\}) } c'_\alpha h_{\alpha, \gamma}+\Sigma_{\beta\in \sigma} b'_\beta \xi_\beta,\,\,c_\alpha, c'_\alpha\in R^1,\,b_\beta, b'_\beta> 0. $$
\end{proof}

Thus, $\Delta_{\sigma', w'}\succeq \Delta_{\sigma, w}$ implies $T_a(B_{\sigma, w}(a))\subseteq T_a(B_{\sigma', w'}(a))$.

\subsection{Projection onto $T_a(X)$}

For a real $n$ dimensional vector space $E^n$, we study the duality between $E^n$ and its dual $(E^n)^*$. Let $\{\textbf{e}_1, \cdots, \textbf{e}_n\}$ be a basis of $E^n,$ and let $\{\textbf{e}^*_1, \cdots, \textbf{e}^*_n\}$ be its dual basis, i.e.,
$$ <\textbf{e}^*_i, \textbf{e}_j>=\delta_{ij},\,\,\,\,\forall , j=1, 2, \cdots, n.$$
Thus, $\forall\, x^*=\sum^n_{i=1}x_i \textbf{e}^*_i, \, \forall\, y=\sum^n_{i=1}y_j \textbf{e}_j$
$$ <x^*, y>=\Sigma^n_{i=1} x_i y_i.$$

It is well known that in a finite dimensional space, all norms are equivalent. Thus
for the given basis and its dual basis one may assign the norms of $y$ and $x^*$ on $E^n$ and $(E^n)^*$ by
\begin{equation*}
\|y\|=(\Sigma^n_{i=1}|y_i|^2)^{1/2}\,\,\mbox{and}\,\,\|x^*\|=(\Sigma^n_{i=1}|x_i|^2)^{1/2}.
\end{equation*}
respectively.

Let
$$ n_0=\Sigma_{\alpha\in \iota} m_\alpha d_{i_\alpha} \textbf{e}_{i_\alpha}^*,\,\,\,\tau_{\sigma,w}=\Sigma_{\beta\in \sigma}w_\beta d_\beta \textbf{e}_\beta^*.$$
The normal direction of the cell $\Delta_{(\sigma, w)}$ reads as
\begin{equation}
n=n_0+\tau_{\sigma,w}.
\end{equation}

Let us denote $span\{T_a(B_{\sigma, w}(a)) \}$ by $X_{\sigma, w}$, and its dual by $X_{\sigma, w}^*$.
\vskip 0.5cm

Given a co-vector $p=\Sigma_{i=1}^n p_i \textbf{e}^*_i$, the projection of $p$ onto $X_{\sigma, w}^*$ is the covector:
$p_n=p-\lambda n$, where $n$ is the normal co-vector, and
$$ \lambda=\frac{\Sigma_{\alpha\in \iota} m_\alpha d_{\alpha}p_\alpha+\Sigma_{\beta\in \sigma}w_\beta d_\beta p_\beta}{\Sigma_{\alpha\in \iota}d_\alpha^2 +\Sigma_{\beta\in \sigma}d_\beta^2}.$$

Noticing that $\{h_{\alpha}, \psi^w_\beta,\,\,\,|\,\alpha\in (\iota\backslash\{i_1\}), \beta\in \sigma\}$ is a basis of the linear subspace $X_{\sigma, w}$, where $h_\alpha=h_{\alpha,1}$, we consider the following vector on $X_{\sigma, w}$ induced by $p$:
\begin{equation}
\hat{P}_{(\sigma, w)}p=\Sigma_{\alpha\in \iota}d_\alpha^2 <p, h_\alpha> h_\alpha+\Sigma_{\beta\in \sigma}d_\beta^2 <p, \psi^w_\beta> \psi^w_\beta.
\end{equation}
\vskip 0.5cm

Note that
$$ <n, h_\alpha>=<n, \psi^w_\beta>=0,\,\,\forall\,\alpha\in (\iota\backslash\{i_1\}), \beta\in \sigma,$$
we have
\begin{equation*}
\hat{P}_{(\sigma, w)}p=\Sigma_{\alpha\in \iota}d_\alpha^2 <p_n, h_\alpha> h_\alpha+\Sigma_{\beta\in \sigma}d_\beta^2 <p_n, \psi^w_\beta> \psi^w_\beta.
\end{equation*}
Therefore there exists a constant $C_1>0$ such that
\begin{equation}
\|\hat{P}_{(\sigma, w)}p\|\le C_1\|p_n\|.
\end{equation}
\vskip 0.5cm

Since
$$\|p_n\|^2=\|p\|^2-\lambda^2\|n\|^2,$$
we obtain
\begin{align*}
<p_n,  \hat{P}_{(\sigma, w)}p>&=\Sigma_{\alpha\in \iota}d_\alpha^2 |<p_n, h_\alpha>|^2 +\Sigma_{\beta\in \sigma}d_\beta^2 |<p_n, \psi^w_\beta>|^2 \\
&=\|p_n\|^2+(\lambda-\frac{m_{i_1} p_{i_1}}{d_{i_1}})^2\|n\|^2.
\end{align*}

\vskip 0.5cm
\textbf{Remark}
Let $p\in \partial f(a)$ be a sub-differential of a function $f$. As a constraint sub-differential, $p_n$ is its projection on $X_{\sigma, w}^*$, and then $ \hat{P}_{(\sigma, w)}p$
is a pseudo gradient vector on $X_{\sigma, w}$, i.e., it satisfies
\begin{enumerate}
\item $\|\hat{P}_{(\sigma, w)}p\|\le C_1\|p_n\|,$\\
\item $<p_n,  \hat{P}_{(\sigma, w)}p>\ge \|p_n\|^2.$
\end{enumerate}
This is the motivation of the above construction.

Moreover, we have
$$ \hat{P}_{(\sigma, w)}p=0\,\,\,\forall\,(\sigma, w)\,\,\,\,\Leftrightarrow\,\,\,\, p=0.$$
\vskip 0.5cm

However, $T_a(B_\delta(\sigma, w))$ is a cone. In some directions, it has boundaries. $\forall\, p\neq 0$. In order to make the the projection of $p$ onto $X_{\sigma, w}$ points inside $\triangle_{\sigma, w}$, we should modify the projection as follow:
\begin{equation}
P_{(\sigma, w)}p=\Sigma_{\alpha\in \iota}d_\alpha^2<p, h_\alpha> h_\alpha+\Sigma_{\beta\in \sigma}d_\beta^2<p, \psi^w_\beta>_+ \psi^w_\beta.
\end{equation}
where
\begin{equation*}x_+=\left\{\begin{array} {l}
x \,\,\,\,\mbox{if}\,\,\,x\ge 0, \\
0 \,\,\,\,\mbox{if}\,\,\,x= 0.
\end{array}\right.
\end{equation*}

After the modification, the pseudo gradient vector $P_{(\sigma, w)}p$ is pushed down to the lower level cell $\Delta_{\sigma_0, w_0}\preceq \Delta_{\sigma, w},$ with
$$ <p, \psi^w_\gamma>\le 0,\,\,\,\forall\,\gamma\in \sigma\backslash\sigma_0,\,\mbox{and}\,<p, \psi^w_\beta> > 0,\,\,\forall\,\beta\in \sigma_0.$$

\vskip 0.5cm

\begin{theorem} $\forall\, p\in \mathbb{R}^n,$
$$ P_{(\sigma, w)} p=0 \,\,\forall\, (\sigma, w)\,\Leftrightarrow\, p=\lambda(n_0+\Sigma_{\beta\notin \iota} c_\beta d_{\beta}\textbf{e}_{\beta}^*), $$
where $\lambda\ge 0$ and $ |c_\beta|\le 1,\,\,\forall\, 1\le \beta\le n-k.$
\end{theorem}
\begin{proof}
$"\Leftarrow".$

The case $\lambda=0$ is trivial, we may assume $\lambda>0$. It is sufficient to show
$$ <p, h_\alpha>=0,\,\,\,\forall\, \alpha\in \iota\,\,\mbox{and}\,\,\,<p, \psi^w_\beta>\le 0\,\,\,\forall\,\beta\notin \iota,\,\forall\,w.$$
In fact,
$$ <p, h_\alpha>=\lambda<n_0, h_\alpha>=\lambda<\Sigma_{i\in \iota}m_id_i \textbf{e}_i^*, \frac{m_\alpha \textbf{e}_\alpha} {d_\alpha}-\frac{m_{i_1} \textbf{e}_{i_1}}{d_{i_1}}>=0,\,\,\forall\, \alpha\in \iota.$$
and
$$ <p, \psi^w_\beta>=\lambda<n_0+\Sigma_{j\notin \iota}w_jd_j \textbf{e}_j^*, \frac{w_\beta \textbf{e}_\beta} {d_\beta}-\frac{m_{i_1}\textbf{e}_{i_1}}{d_{i_1}}>=\lambda(c_\beta w_\beta-1)\le 0,\,\,\forall\,\beta\notin \iota\,\forall\,w.$$
\vskip 0.5cm

$"\Rightarrow"$ Now, we assume $\forall\, \sigma$ (including $\sigma=\emptyset), \forall\, w:\sigma \to \{\pm 1\}, P_{(\sigma, w)} p=0$.
These imply:
\begin{equation}
<p, h_\alpha>=0,\,\,\,\forall\, \alpha\in \iota,
\end{equation}
and
\begin{equation}
<p, \psi^w_\beta>\le 0\,\,\,\forall\,\beta\notin \iota,\,\forall\,w.
\end{equation}

From (4.9), it follows
$$ \frac{m_\alpha p_\alpha}{d_\alpha}-\frac{m_{i_1} p_{i_1}}{d_{i_1}}=0,\,\,\forall\, \alpha\in \iota.$$
Let $ \lambda=\frac{m_{i_1} p_{i_1}}{d_{i_1}}$.  We obtain
$$ p_\alpha=\lambda m_\alpha d_\alpha,\,\,\forall\, \alpha\in (\iota\backslash\{1\}).$$

From (4.10), we obtain:
$$ \frac{w_\beta p_\beta}{d_\beta}\le \lambda,\,\,\forall\, \beta\notin \iota,\,\,\forall\,w: \sigma\to \{+1, -1\}$$
and then $\lambda\ge 0$.
\vskip 0.5cm

If $\lambda=0$, then $p_{1}=0$ and then $p_{\alpha}=0, \alpha\ge 2$ and $p_{\beta}=0\, \forall\, \beta\notin \iota$.
i.e., $p=0$. The proof is done.

Otherwise, $\lambda>0,$ let
$$ c_\beta=\frac{p_{\beta}}{\lambda d_{\beta}}.$$
Then $ |c_\beta|\le 1, $ and
\begin{align*}
p &= \Sigma_{\alpha\in \iota} p_{\alpha} \textbf{e}_{\alpha}+\sum_{\beta\notin \iota} p_\beta \textbf{e}_\beta\\
&=\lambda(\Sigma_{\alpha\in \iota} m_\alpha d_{\alpha} \textbf{e}_{\alpha}+ \Sigma_{j_\beta\notin \iota} c_\beta d_{\beta}\textbf{e}_{\beta})\\
&=\lambda (n_0+ \Sigma_{\beta\notin \iota} c_\beta d_{\beta}\textbf{e}_{\beta}).
\end{align*}
The proof is complete.
\end{proof}
\vskip 0.5cm

When we want to specify the component $P_{(\sigma, w)}p$ being at the point $a$, we denote it by $P(a)_{(\sigma, w)}p$, and define
\begin{definition} $\forall a\in \Gamma_{\iota, m}\subset X,\, \forall\, p\in R^n,$ the collection of $3^{n-k}$ vectors
$$ P(a) p=\{ P(a)_{(\sigma, w)}p\,|\, \forall\,(\sigma, w)\,\mbox{including}\,\sigma=\emptyset\}$$
is called the projection of $p$ onto $T_a(X)$, where the component $P_{(\sigma, w)}p$ is the vector in $T_a(B_{(\sigma, w)}(a))$.
\end{definition}
\vskip 0.5cm

It is important to note

\begin{enumerate}
\item Generally, $P(a)(-p) \neq -P(a)p.$
In particular, to those $p$ satisfying $P(a)p=0,$ one may have $P(a)(-p)\neq 0$!\\
\item $P(a)_{(\sigma, w)}p\in T_a(B_{(\sigma, w)}(a)).$\\
\item The mapping $p\to P(a) p$ is continuous.
\end{enumerate}
\subsection{Critical point theory on $X$}

Let $f: \mathbb{R}^n\to \mathbb{R}^1$ be a locally Lipschitzian function, the sub-differential $\partial f(x)$ is defined in the Clarke sense. The set valued mapping:
$$ x \mapsto \partial f(x) $$
is convex closed valued and upper semi continuous (u.s.c), i.e., $\forall\, x_0\in \mathbb{R}^n,$
$$x_k\to x_0, \,u_k\in \partial f(x_k)\,\, \mbox{and}\, u_k\to u_0\, \mbox{imply} \,\,u_0\in \partial f(x_0).$$

To the constraint problem, we introduce
\begin{definition} Let $\tilde{f}=f|_X.\,\,a\in X$ is called a critical point of $f$, if $\exists\, p\in \partial f(a)$ such that $P(a)p=0.$ In other words,
$$ 0\in P(a)\partial f(a).$$
Let $K$ be the set of all critical points of $\tilde f$, it is called the critical set. Usually, we write $K_c=K\cap f^{-1}(c),\,\forall\, c\in R^1.$
\end{definition}
\vskip 0.5cm

\begin{lemma}
The set valued mapping $x \mapsto P(x) \partial f(x)$ is u.s.c. i.e, each of its components $P(x)_{(\sigma, w)}\partial f(x)$ is u.s.c. $\forall\,(\sigma, w)$. In other words,
$\forall\,(\sigma, w)$, if $a_k\to a$ in $\Delta_{\sigma, w}$, $p_k\in P(a_k)_{(\sigma, w)}\partial f(a_k)$ and $p_k\to p,$ we have $p\in P(a)_{(\sigma, w)}\partial f(a)$.

Consequently, the critical set $K$ is closed.
\end{lemma}
\begin{proof}
In fact, from the u.s.c. of $x \mapsto \partial f(x)$ and the continuity of $x \mapsto P(x)p$:
$$ <p_k, h_\alpha>\to <p, h_\alpha>,\,\mbox{and}\,\,<p_k, \psi^w_\beta>\to <p, \psi^w_\beta>,$$
the u.s.c. of $x \mapsto P(x) \partial f(x)$ follows.
The closeness of $K$ follows directly from the u.s.c. of $P(x) \partial f(x)$.
\end{proof}
\vskip 0.5cm

 We define
\begin{definition}$\forall\, a\in \Gamma_{\iota, m},\,\forall\, (\sigma, w)$, set
$$ \lambda (a)=\min_{u\in \partial f(x)}\{max_{(\sigma, w)}\{\|p\|\,|\,p\in P_{(\sigma, w)}u\}.$$
\end{definition}

By definition, $a\in X$ is a critical point of $\tilde f$ if and only if $\lambda(a)=0.$

%\vskip 0.5cm

\begin{lemma}For any $a\notin K,\,\,\forall\,\epsilon>0, \exists\,v_0\in T_a(X)$ and $\delta>0$, satisfying $\|v_0\|=1$ and
$$ f(y-tv_0)\le f(y)-t(r-\epsilon) \,\,\,\forall\, y\in B_\delta(a),\,\,\forall\,t\in [0, \delta).$$
where $r=\lambda (a)$
\end{lemma}
\begin{proof}

Since $a\notin K$, $\lambda (a)>0$. There must be $(\sigma, w)$ such that $r=\lambda (a)=\min_{p\in\partial f(a)}\|P(a)_{(\sigma, w)}p\|.$
Let $v=\Sigma_{\alpha\in \iota}d_\alpha^2 \xi_\alpha h_\alpha+\Sigma_{\beta\in \sigma}d_\beta^2\eta_\beta \psi^w_\beta.$ Then
\begin{align*}
&\lim_{t\downarrow 0, y\to a} t^{-1}[f(y-tv)-f(y)]=f^o(a, -v)\\
&=\max_{p\in \partial f(a)} <p, -v>
=-\min_{p\in \partial f(a)}\Sigma_{\alpha\in \iota}d_\alpha^2\xi_\alpha <p, h_\alpha>+\Sigma_{\beta\in \sigma}d_\beta^2\eta_\beta <p, \psi^w_\beta>
\end{align*}
where $f^o$ is the Clarke's directional derivative.

Now, $\partial f(a)$ is a closed convex set, there exists $p_0\in \partial f(a)$ achieves the minimum $r^2$ of the quadratic function:
$$ \|P(a)_{(\sigma, w)}p\|^2=\Sigma_{\alpha\in \iota} d_\alpha^2<p, h_\alpha>^2+\Sigma_{\beta\in \sigma}d_\beta^2<p, \psi^w_\beta>^2.$$
This implies
$$ \min_{p\in \partial f(a)}\Sigma_{\alpha\in \iota}d_\alpha^2 <p, h_\alpha><p_0, h_\alpha>+\Sigma_{\beta\in \sigma}d_\beta^2<p, \psi^w_\beta><p_0, \psi^w_\beta> \ge \|P(a)_{(\sigma, w)}p_0\|^2=r^2.$$

Let $v_0=\frac{1}{r}\Sigma_{\alpha\in \iota}d_\alpha^2 <p_0, h_\alpha>h_\alpha+\Sigma_{\beta\in \sigma}d_\beta^2<p_0, \psi^w_\beta>\psi^w_\beta.$ Then $\|v_0\|=1$, and

\begin{align*}
&\lim_{t\downarrow 0, y\to a} t^{-1}[f(y-tv_0)-f(y)]\\
&= -\frac{1}{r}\min_{p\in \partial f(a)}\Sigma_{\alpha\in \iota}d_\alpha^2 <p, h_\alpha><p_0, h_\alpha>+\Sigma_{\beta\in \sigma}d_\beta^2<p, \psi^w_\beta><p_0, \psi^w_\beta>\\
&\le -r.
\end{align*}

It follows
$$ f(y-tv_0)\le f(y)-t(r-\epsilon)\,\,\,\,\,\,\forall\, y\in B_\delta(a)\,\forall\, t\in [0, \delta).$$
\end{proof}
\vskip 0.5cm

\textbf{Remark}   In the terminology of weak slope due to Corvellec, Degionanni, Marzocchi\cite{refCDM}, we introduce the mapping $H: B_\delta(a)\times [0, \delta)\to X$ as follow
$$ H(y, t)=y-tv,$$
where $v=v_0$ with $r-\epsilon$ in replacing $r$, then
$$ \| H(y, t)-y\|= t.$$
The weak slope is defined by
$$ |d\tilde{f}|(a)=\sup\{r\in [0, \infty) \,|\, f( H(y, t))\le f(y)-tr\}.$$

In this sense,
$$ |d\tilde{f}|(a)=0\,\,\,\Leftrightarrow\,\,\,\lambda(a)=0.$$

\vskip 0.5cm

Based on the function $\lambda,$ by standard procedure we can construct a pseudo-gradient vector field $v$ and the associate pseudo-gradient flow $\eta: (X\backslash K)\times R^1 \to X $ with respect to $\tilde{f}$ on $X$, see Chang\cite{refC1}, \cite{refC2}, Rabinowitz\cite{refR}. Then the following deformation theorem holds.

\begin{theorem}(Deformation) Let $c\in \mathbb{R}^1,\, K_c=K\cap \tilde{f}^{-1}(c)$, and $N\subset X$ is a neighborhood of $K_c,$ then $\forall\, \epsilon_0>0,\,\,\exists\, \epsilon\in (0, \epsilon_0)$, and a deformation $\eta: [0,1]\times X\to X$ satisfying
\begin{enumerate}
\item $\eta(0, x)=x\,\,\,\,\forall\, x\in X$,\\
\item $\eta(t, x)=x\,\,\,\,\forall\, x\notin \tilde{f}^{-1}[c-\epsilon_0, c+\epsilon_0]$,\\
\item $ \forall\, t\in[0,1], \eta(t, \cdot): X\to X$ is a homeomorphism,\\
\item $\eta(1, \tilde{f}_{c+\epsilon}\backslash N) \subset \tilde{f}_{c-\epsilon}$, where $\tilde{f}_b $ is the level set of $\tilde{f}$ below or equal to $b$,\\
\item If $K_c=\emptyset$, then $\eta(\tilde{f}_{c+\epsilon}) \subset \tilde{f}_{c-\epsilon}$.
\end{enumerate}
\end{theorem}
Liusternik- Schnirelmann theory is applied to study the multiplicity of the critical points for even functions on the the symmetric piecewise linear manifold $X$. We use the genus version of the theorem due to Krasnoselski, see Chang\cite{refC2}, and Rabinowitz\cite{refR}. Let $A\subset\mathbb{R}^n$ be a symmetric set, i.e., $-A=A,$ satisfying $\theta\notin A$. An integer valued function, which is called the genus of $A$, $\gamma: A \to Z_+\cup\{+\infty\}$ is defined:
\begin{equation*} \gamma(A)=\left\{\begin{array} {l}
0\,\,\,\,\,\,\,\mbox{if}\,\,A=\emptyset,\\
min\{k\in Z_+\,|\, \exists\,\mbox{odd continuous}\,h: A\to S^{k-1}\}
\end{array}\right.
\end{equation*}
Genus is a topological invariant.
\vskip 0.5cm

\begin{theorem} Suppose that $f$ is a locally Lipschitzian even function on $\mathbb{R}^n$, then
\begin{equation*}
c_k=\inf_{\gamma(A)\ge k} Max_{x\in A}\tilde{f}(x)
\end{equation*}
are critical values of $\tilde{f}, \,k=1,2, \cdots n.$ They satisfy
$$c_1\le c_2 \le \cdots \le c_n.$$

Moreover, if
\begin{equation*}
c=c_{k+1}=\cdots =c_{k+l},\,\, 0\le k\le k+l \le n,
\end{equation*}
then $\gamma(K_c)\ge l.$
\end{theorem}
\vskip 0.5cm

A critical value $c$ is said of multiplicity $l$, if $\gamma(K_c)=l.$

Thus, we have
\begin{theorem} There are at least $n$ critical points $\phi_k,\,\,k=1,2, \cdots, n$ of $\tilde{I}$ such that
$\phi_k\in K_{c_k}$. Moreover, counting multiplicity, $\tilde{I}$ has at least $n$ critical values.
\end{theorem}

\vskip 0.5cm
\subsection{Connection between $K$ and $S(G)$}

\begin{theorem} The critical set $K$ of $I$ is the same with the set of all eigenvectors $S(G)$ on $G$, i.e., $K=S(G)$.
\end{theorem}
\begin{proof} Let $a\in \Gamma_{\iota, m}.$ $a=(a_1, \cdots, a_n)\in K$ if and only if $\exists\, p\in \partial I(a)$ such that $P(a)p=0$. However, $p\in \partial I(a)$ means
that $\exists\, z_{ij}\in Sgn(x_i-x_j),\,z_{ji}=-z_{ij}$ such that $p_i=\Sigma_{j\sim i}x_{ij},\,\forall\,i.$

According to Theorem 4.2, $P(a)p=0$ means that $ \exists\, c_i$ with $|c_i|\le 1$, such that
$$ p=\mu(n_0+\Sigma_{j_\beta\notin \iota} c_\beta d_{j_\beta} \textbf{e}_{j_\beta}),$$
with
$$\mu=\frac{(p, n_0)}{\|n_0\|^2}.$$
i.e.,
\begin{equation*}\Sigma_{j\sim i}z_{ij}(a)= \left\{\begin{array} {l}
\mu d_i sgn (a_i),\,\,\,\,\,\mbox{if}\,i\in \iota, \\
\mu d_i c_i,\,\,\,\,\,\mbox{if}\,i\notin \iota.
\end{array}\right.
\end{equation*}

 By Definition 2.3, for $a\in S(G)$ if and only if
\begin{equation*}\Sigma_{j\sim i}z_{ij}(a)\in \mu d_iSgn(a_i).
\end{equation*}
Therefore, $K=S(G)$.
\end{proof}
\vskip 0.5cm

\begin{definition} An eigenvalue $\mu$ of (1.8) is of multiplicity $l$, if $\gamma(S(G)\bigcap I_\mu)=l.$
\end{definition}

Following Corollary 2.5, Theorem 4.10 and Theorem 4.11, we immediately obtain

\begin{theorem} $\forall\, \mu\in [0, 1], \, K_\mu=S(G)\cap I_\mu$, i.e., the critical set with critical value $\mu$ is  the set of eigenvectors with eigenvalue $\mu$. Consequently, there are at least $n$ eigenvectors $\phi_k$ of the eigenvector system such that
$$I(\phi_k)=c_k,$$
where
\begin{equation}
c_k= \inf_{\gamma(A)\ge k} Max_{x\in A}\tilde{I}(x),\,\,\,k=1,2, \cdots, n.
\end{equation}
Moreover, if $\mu=c$ and (4.11) holds, then the multiplicity of $\mu$ is greater than or equal to $l$,
and then counting multiplicity, (1.8) has at least $n$ eigenvalues.
\end{theorem}
The above theorem can be seen as the counterpart of the multiplicity theorem in linear spectral theory.
\vskip 0.5cm

In this sense, the set of eigenvectors associate to an eigenvalue of multiplicity 1 may not be a single vector, but a symmetric set of eigenvectors with genus 1.

In contrast to the linear spectral theory of graphs, the system (2.4) is nonlinear, neither algebraic, there is no way to define the algebraic multiplicity for eigenvalues. The above defined multiplicity is geometric or topological.
\vskip 0.5cm

\textbf{Remark}

Following Theorem 3.9, Theorem 4.11, and the fact that there are totally $3^n$ closed cells for a graph $G$ with $n$ vertices, we conclude: the spectrum for $\Delta_1$ on graphs is discrete, because the number of distinct eigenvalues is no more than $\frac{3^n}{2}$.

Now, we can arrange eigenvalues of $\Delta_1(G)$ in increasing order:
$$ 0=\mu_1\le \mu_2\le \cdots \le 1.$$
\vskip 0.5cm

\textbf{Question} We have also arranged some critical values of $\tilde{I}$ in increasing order as in Theorem 4.13.

According to Theorem 4.11, the set of critical values of $\tilde{I}$ is the same as that of eigenvalues of $\Delta_1(G)$, then we ask:
Is there any eigenvalue $\mu$, which is not in the sequence: $\{c_1, c_2, \cdots, c_n\}$?

\section{Further results on eigenvalues and eigenvectors}

\subsection{Elementary facts}

Let $x$ be an eigenvector of $\Delta_1(G)$ with eigenvalue $\mu$.
\begin{theorem}
For $\mu=1$ if and only if any nodal domain of $x$ consists of a single vertex.
\end{theorem}
\begin{proof}
$"\Rightarrow"$ Assume $\mu=1$

1. Suppose $x_i\neq 0$ for some $i.$ Look at the equation:
$$ \Sigma_{j\sim i} z_{ij} =d_i sgn (x_i).$$
Since $|z_{ij}|\le 1$, and $d_i$ is the number of $j$, which is adjacent to $i$, i.e., $j\sim i,$ we have
$$ z_{ij}=sgn (x_i),\,\,\,\,\,\,\forall\, j\sim i.$$
2. Suppose there is a nodal domain $D^+_\alpha$ (similarly $D^-_\beta$),which consists of more than one vertices. Say, there are $i$ and $j$ with $j\sim i.$ According to the previous conclusion, we have
$$ z_{ij}=sgn (x_i) =1,$$
and
$$ z_{ji}=sgn (x_j) =1.$$
But,
$$z_{ij}=-z_{ji}.$$
This is a contradiction.
\vskip 0.5cm

$"\Leftarrow"$  Suppose $D^+_\alpha$ consists of a single vertex. According to $(3.6)$,
$$ p^+_\alpha=\mu \delta^+_\alpha.$$
By the assumption, there is no inter-domain connection in $D^+_\alpha$, it must be
$$ p^+_\alpha= \delta^+_\alpha.$$
Therefore $\mu=1.$
\end{proof}

\vskip 0.5cm

\begin{theorem}
For $\mu<1$ if and only if any nodal domain of $x$ contains at least a pair of adjacent vertices.
\end{theorem}
\begin{proof}
$"\Leftarrow"$ It follows from the above lemma.

$"\Rightarrow"$ For any nodal domain $D^{\pm}_\tau$, we obtain from $(3.6)$,
$$ p^{\pm}_\tau=\mu \delta^{\pm}_\tau <\delta^{\pm}_\tau.$$
This means that besides those vertices adjacent outside the nodal domain, there is at least a pair of adjacent vertices inside.
\end{proof}
\vskip 0.5cm

\begin{corollary} If $G$ is a connected graph, and there exists a vertex $i_0$ such that
$$ d_{i_0}=\Sigma_{i\neq i_0} d_i,$$
then $I$ has only two eigenvalues: $0$ and $1$, i.e., $I(S(G))=\{0, 1\}$.
\end{corollary}
\begin{proof}
May assume $i_0=1$. We have $d_2=\cdots =d_n=1$ and $d_1=n-1.$
Suppose the conclusion is not true, i.e., there exists an eigenvalue $\mu\in (0,1)$. According to Theorem 5.2, any nodal domains contain more than one vertex. However, $G$ is connected, it must contain $i_0$.
Thus we have only one nodal domain, and then all terms on the RHS of the system $(1.8)$ have no distinct sign. It contradicts with Theorem 2.6.
\end{proof}
\vskip 0.5cm

In fact, we find the solutions as follow:

\begin{align*}
&\mu_1=0, \,\,\,\phi_1= \frac{1}{2(n-1)}\Sigma^n_{i=1} \textbf{e}_i,\\
&\mu_2=\cdots=\mu_n=1,\\
&\phi_2=\frac{1}{2(n-1)}[\textbf{e}_1-\Sigma^n_{i=2}\textbf{e}_i],\\
&\phi_k=\frac{1}{2}[\textbf{e}_2-\textbf{e}_{k}],\,\,\,\,k=3,\cdots, n.
\end{align*}
\vskip 0.5cm

\begin{theorem}
For $\mu=0$ being a simple eigenvalue, i.e., $x_0=\frac{1}{\Sigma^n_{i=1}d_i}\textbf{1}$ is the unique eigenvector with respect to $0$ if and only if $G$ is connected, where $\textbf{1}=(1,1, \cdots,1)$.
\end{theorem}
\begin{proof}
$"\Leftarrow"$ Suppose $G$ is connected, let $\xi$ be an eigenvector with eigenvalue $0$. From Corollary 2.5,
$I(\xi)=\Sigma_{j\sim i}|\xi_i-\xi_j|=0,\,\,\forall\,i$, it implies
$$ \xi_i=\xi_j, \,\,\,\forall\, j\sim i.$$
However, $G$ is connected, all $\xi_i$ must equal. From $\xi\in X,$ we have $\xi=x_0.$

$"\Rightarrow"$ Suppose that $G$ is not connected, say there are at least two connected components $G_1$
and $G_2$. We define
$$ \xi_k=\frac{1}{D_k}\Sigma_{i\in G_k}\textbf{e}_i,$$
where $D_k= \Sigma_{i\in G_k} d_i\,\, k=1,2.$ Then we have
$$ I(\xi_1)=I(\xi_2)=0.$$
$\xi_1$ and $\xi_2$ are distinct eigenvectors with the same eigenvalue $0$.
\end{proof}
\vskip 0.5cm

\begin{corollary} If $G$ consists of $r$ connected components $G_1, \cdots, G_r$, then the eigenvalue $\mu=0$ has multiplicity $r$. i.e., all the eigenvectors with respect to $0$ consist of a critical set $A$ with $\gamma(A)=r.$
\end{corollary}
\begin{proof} We have the following $r$ pairs with eigenvalue $0$,
$$ \xi_k=\pm \frac{1}{D_k}\Sigma_{i\in G_k}\textbf{e}_i,$$
where $D_k= \Sigma_{i\in G_k} d_i\,\, k=1,2,\cdots, r.$
Let $T=span\{\xi_1,, \cdots, \xi_r \}$ and $B=T\cap X$. It is homeomorphic to $S^{r-1}$. By definition, $B\subset A$, therefore
$$\gamma(A)\ge \gamma(B)=r.$$

It remains to show: $\gamma(A)\le r.$ If not, we assume $\gamma(A)> r.$

Let $G_k=(V_k, E_k), Y_k=R^{|V_k|},$ where $|V_k|$ is the cardinal number of $V_k$, let $Z_k=Y_k\cap \xi_k^\bot$ be the orthogonal complement of $\xi_k$ in $Y_k$, and let $P_k: R^k\to Y_k$ be the orthogonal projection, $k=1,2, \cdots, r .$ Since $G_1, \cdots G_r$ are components of $G,$ we have
$$ I(x)=\sum^r_{k=1} I(P_k x).$$
Note that
$$ T^\bot=\bigoplus_{k=1}^r  Z_k.$$
then $$dim(T^\bot)=n-r.$$
Since we assume $\gamma(A)> r,$ then $A\cap T^\bot\neq \emptyset$, according to the intersection property, see \cite{refR}, \cite{refC2}. That is $\exists\, x_0\in T^\bot\cap A$, from
$$ \sum^r_{k=1} I(P_k x_0)=I(x_0)=0,$$
it follows,
$$I_k(x_0)=0,\,\,\,k=1,2, \cdots, r.$$
According to Theorem 5.4, $P_k x_0=\xi_k,\,\,k=1, 2, \cdots, r,$ we obtain
$$ x_0=\Sigma^r_{k=1}\xi_k \in T.$$
This is a contradiction..
\end{proof}
\vskip 0.5cm

An estimate of the nontrivial eigenvalues is obtained.
\begin{theorem}
If the eigenvalue  $0<\mu<1$, then
$$ \frac{2}{\Sigma^n_{i=1} d_i}\le \mu \le \frac{n-2}{n-1}.$$
\end{theorem}
\begin{proof} We may assume that $G$ is connected. Let $x$ be the eigenvector with respect to $\mu$, and let $D^+$ be a nodal domain of $x$. According to Theorem 5.2, it contains at least $2$ adjacent vertices. Thus, following the notations in subsection 3.1,
$$ \delta^+\ge p^+ + 1.$$
Now,
$$ \mu=\frac{p^+}{\delta^+}\le 1-\frac{1}{\delta^+},$$
We claim
$$ \delta^+\le n-1.$$
For otherwise, $D^+=V$, and then $\mu=0$. This is a contradiction. We obtain
$$ \mu\le 1-\frac{1}{n-1}.$$

As to the left inequality, we notice that the normal eigenvector $x\in \pi \cap X,$ is of the form $(3.13)$, there must be vertices $k$ and $l$ such that
$ x_k x_l<0$, with either $k\sim l$ or $k\sim D^0, l\sim D^0$, where $D^0$ is the null set of $x$. Let $\delta =\sum_{i\notin D^0} d_i,$ then by Corollary 2.5,
$$ \mu=I(x)=\Sigma_{j\sim i}|x_i-x_j|\ge  2\delta^{-1}.$$
Since $\delta \le \Sigma_{i=1}^n d_i, $ the conclusion follows.
\end{proof}
\vskip 0.5cm

\subsection{Alternative characterization of the mountain pass point and the second eigenvalue}

Given an open $k-1$ cell $\Delta=\Gamma_{\iota, m}$ with $(\iota, m)\in I_k\times M_k$, the index subsets
$D^\pm, D^0$ are all determined:
$$ D^\pm=\{ i\in \iota\,|\, \pm m(\iota)_i>0\},\,\, D^0=\{i\notin  \iota\}.$$
And then for a graph $G$, the nodal domains $D^\pm_\gamma(x)$ are invariant$\,\forall\, x\in \Delta$.
\vskip 0.5cm

Now we improve Corollary 2.5 to the following:
\begin{lemma} For a pair $(\mu, x)\in R^1\times X$, if there exist $z_{ij}=z_{ij}(x)\in Sgn(x_i-x_j)$, satisfying $z_{ij}=-z_{ji}$, and
$$ \Sigma_{j\sim i} z_{ij}(x)=\mu d_i sgn(x_i),\,\,\,\,\forall\,i\in D^+\cup D^-,$$
then $\mu=I(x)$.
\end{lemma}
\begin{proof} Since
$$ \Sigma_{j\sim i} z_{ij}(x) x_i=\mu d_i |x_i|,\,\,\,\,\forall\,i\in D^+\cup D^-,$$
and
$$ \Sigma_{j\sim i} z_{ij}(x) x_i=0.\,\,\,\,\forall\,i\notin  D^+\cup D^-,$$
Again by summation, we obtain
$$ \Sigma_{j\sim i}|x_i-x_j|=\mu\Sigma_{i=1}^n d_i|x_i|=\mu.$$
\end{proof}
\vskip 0.5cm

\begin{lemma}Let $\Delta=\Gamma_{\iota, m}$ be a $k-1$ dimensional open cell in $X$. Assume that $\xi\in \Delta$ attains the minimum of $I$ on $\Delta$, and
$$ c_\Delta=r(E^+-E^-),\,\,\delta=\delta^++\delta^-,\,\, r=\delta^{-1}$$
where $E^\pm=\Sigma_{i\in D^\pm} e_i$,  and $\delta^\pm=\Sigma_{i\in D^\pm} d_i$. Then $$I(\xi)=I(c_\Delta).$$
\end{lemma}
\begin{proof}
$1^o$ Since $\xi=(\xi_1.\cdots, \xi_n)$ attains the minimum of $\tilde{I}$ on $\Delta$ and $\Delta$ is open, we have
$$ \partial I(\xi)=\mu d_i sgn(\xi_i),\,\,\forall\, i\in D^+(\xi)\cup D^-(\xi),$$
for some $\mu\in R^1,$ i.e., $\exists\, z_{ij}(\xi)\in Sgn(\xi_i -\xi_j)$, satisfying $ z_{ji}(\xi)= - z_{ij}(\xi)$ and
$$\sum_{j\sim i} z_{ij}(\xi)=\mu d_i sgn(\xi_i),\,\,\forall\, i\in D^+(\xi)\cup D^-(\xi).$$
From the previous lemma, we have $\mu=I(\xi)$.

$2^o$ Let $\xi^t=(1-t)\xi+t c_\Delta,\,t\in [0, 1]$, then $\xi^t=(\xi^t_1,\cdots, \xi^t_n)$ satisfies
\begin{equation*}\xi^t_i= \left\{\begin{array} {l}
(1-t)\xi_i+tr,\,\,\,\,\,\mbox{if}\,i\in D^+, \\
(1-t)\xi_i-tr,\,\,\,\,\,\mbox{if}\,i\in D^-,\\
(1-t)\xi_i\,\,\,\,\,\,\,\mbox{if}\,i\in D^0.
\end{array}\right.
\end{equation*}
It is easily seen
$$ Sgn(\xi^t_i-\xi^t_j)=Sgn(\xi_i-\xi_j)\,\,\,\forall\, t\in [0, 1), \forall\,i\in D^+\cup D^-,\,\forall\, j.$$
and
$$ sgn(\xi^t_i)=sgn(\xi_i),\,\,\,\forall\, t\in [0, 1],\, \forall\,i\in D^+\cup D^-.$$
$\forall\, i\in D^+_\alpha$, from
$$Sgn(\xi^1_i-\xi^1_j)= 1,\,\,\,\,\forall\,j\in D^-\cup D^0, $$
$$Sgn(\xi^1_i-\xi^1_j)=[-1, 1]\,\,\,\,\,\,\,\forall\,j\in D^+_\alpha,$$
and the relation: $z_{ij}=-z_{ji}$, it follows
$$ Sgn(\xi^t_i-\xi^t_j)\subset Sgn(\xi^1_i-\xi^1_j),  \,\,\forall\, t\in [0,1).$$
We take  $z_{ij}(\xi^1)=z_{ij}(\xi)\in Sgn(\xi^1_i-\xi^1_j)$, it follows
$$ \mu d_i sgn(\xi^1_i) =\mu d_i sgn(\xi_i)=\Sigma_{j\sim i}z_{ij}(\xi)\in \Sigma_{j\sim i}Sgn(\xi^1_i-\xi^1_j), \,\,\,\,\forall\, i\in D^+\cup D^-.$$
According to lemma 5.7,
$$ I(\xi)=\mu= I(c_\Delta).$$
\end{proof}
\vskip 0.5cm

\begin{lemma} The subset $\pi \subset X$ is closed.
\end{lemma}
\begin{proof} Note that the numbers $\delta^\pm(x)$ as well as $\delta^0(x)$ are constants on each open cell
$\Delta\subset \pi$. However, if a sequence $\{ x^{(m)}\}\subset \Delta$ tends to a point $x^{(0)}$ on the boundary of $\Delta$, then the number gained in $\delta^0(x^{(0)})$ are those lost from $\delta^+(x^{(0)})$ and $\delta^-(x^{(0)})$. Therefore
$$|\delta^+(x^{{m}})-\delta^-(x^{{m}})| \le \delta^0(x^{{m}}),$$
implies
$$|\delta^+(x^{{0}})-\delta^-(x^{{0}})| \le \delta^0(x^{{0}}).$$
This means that $x^{(m)}\subset \pi$ implies $x^{(0)}\in \pi.$
\end{proof}
\vskip 0.5cm

\begin{lemma}Let $m=Min_{x\in \pi} I(x)$ Then there exists an element of the form
$$ \phi =\delta^{-1}(E^+-E^-),\,\,\,\,\delta=\delta^++\delta^-,$$
such that $m=I(\phi)$.
\end{lemma}
\begin{proof}
Since $\pi$ is closed, it is compact, there exists $x_0\in \pi$ such that $m=I(x_0).$ If $m<1$. then $x_0\in X\backslash V$. Following the notations used in section 4,
$ X=\cup_{k=1}^n S_k$ is a disjoint union, and also
$$ S_k=\cup_{(\iota, m)\in I_k\times M_k}\Gamma_{\iota, m}.$$

All $\Gamma_{\iota, m}$ are open sets except $k=1,$  there must be a unique
$(\iota,  m)\in I_k\times M_k, k>1,$ such that $x_0\in \Gamma_{\iota, m}$. Now $\Gamma_{\iota, m}$ is open,
after lemma 5.8, we have $\phi$ of the above form such that $I(\phi)=I(x_0)=m.$

Otherwise, $m=1,$ either $x_0\in \Gamma_{\iota, m}$ with $(\iota, m)\in I_k\times M_k,\,k>1$, or $x_0=e_i$ for some $i\in V$. In the previous case, the conclusion follows from the above argument, and in the latter case, one takes $E^+=e_i, E^-=\emptyset,$ i.e., $\phi=\frac{1}{d_i}e_i$.
\end{proof}
\vskip 0.5cm

\begin{lemma} Let $D^\pm$ and $D^0$ be disjoint index subsets of $\{1, 2, \cdots n\}$ with $D^+\cup D^-\cup D^0= \{1, 2, \cdots n\}.$ Let $\phi=(\delta^++\delta^-)^{-1}(E^+-E^-)$ and $\phi_0=\frac{1}{d}(E^++E^-+E^0)$, where $d=\Sigma^n_{i=1}d_i=\delta^+ +\delta^- +\delta^0$. Then there exists a path on $X$ connecting $\phi$ and $\phi_0$ in the level set $I_c$, where $c=I(\phi)$.
\end{lemma}
\begin{proof}
Define
$$\phi^t=(\frac{t}{d-\delta^0}+\frac{1-t}{d})E^+ + (\frac{-t}{d-\delta^0}+\frac{1-t}{d})E^- +\frac{1-t}{d}E^0,\,\,\,\,t\in [0, 1],$$
and
$$\phi_t=g(t)^{-1} \phi^t, $$
where
$g(t)=|\phi^t|$ and $|x|=\Sigma^n_{i=1}d_i|x_i|$.
Thus, $\phi_t\in X$. According to Theorem 3.6, we have
\begin{align*}
I(\phi^t)&=\Sigma_{\alpha, \beta} Z_{\alpha\beta}\frac{2t}{d-\delta^0}+(\Sigma_\alpha Z_\alpha^\downarrow + \Sigma_\beta Z_\beta^\uparrow) \frac{t}{d-\delta^0}\\
&=\frac{t}{d-\delta^0}[\Sigma_{\alpha, \beta} 2Z_{\alpha\beta}+(\Sigma_\alpha Z_\alpha^\downarrow + \Sigma_\beta Z_\beta^\uparrow)]\\
&=t I(\phi)
\end{align*}

Therefore
$$I(\phi_t)=\frac{t}{g(t)}I(\phi),$$
and
\begin{equation*}g(t)= \left\{\begin{array} {l}
\frac{\delta^+ +\delta^0- \delta^-}{d}+\frac{2\delta^-}{d}t,\,\,\,\,\,\mbox{if}\,t\ge t_0:=\frac{d-\delta^0}{2d-\delta^0}, \\
1- \frac{2\delta^-}{\delta^+ + \delta^-}t,\,\,\,\,\,\mbox{if}\,t\le t_0.
\end{array}\right.
\end{equation*}
Noticing
\begin{equation*}(\frac{t}{g(t)})'= \frac{1}{g(t)^2} \left\{\begin{array} {l}
\frac{\delta^+ +\delta^0- \delta^-}{d},\,\,\,\,\,\mbox{if}\,t\ge t_0, \\
1,\,\,\,\,\,\mbox{if}\,t\le t_0.
\end{array}\right.
\end{equation*}
Since $|\delta^+-\delta^-|\le \delta^0$, $I(\phi_t)$ is increasing, we have $I(\phi_t)\le I(\phi)=c$.
 The path $\{\phi_t \,|\, t\in [0, 1]\}$ is in the level set $I_c$.
\end{proof}
\vskip 0.5cm

Recall Theorem 3.12,and lemma 5.9, all eigenvectors with eigenvalues $\mu\neq 0$ lie on the compact subset:
$$ \pi=\{x\in X\,|\, |\delta^+(x)-\delta^-(x)|\le \delta^0(x)\}.$$
Let us define
$$ m=Min\{ I(x)\,|\, x\in \pi\}.$$
and then turn to prove
\vskip 0.5cm
\begin{theorem}
If $G$ is connected, then
$$ c_2=\mu_2=m.$$
\end{theorem}
\begin{proof}
We only need to prove: $0<m=c_2.$ Once it is shown, then $m$ is a positive eigenvalue according to Theorem 4.11.
Since all eigenvectors with positive eigenvalues lie in $\pi$, and $m$ is the minimum on $\pi$, it follows $m\le \mu_2\le c_2.$

We first prove:  $c_2\ge m$. In fact, there exists $a\in K \subset \pi,$ such that $c_2=I(a)$. Therefore
$$ c_2 \ge Min\{I(x)\,|\, x\in \pi\}=m. $$

Next, we verify $c_2\le m$ and $m>0$

$1^o.$ Since $I$ is continuous, and $\pi$ is compact, the minimum of $I$ on $\pi$ is achieved. Let it be $\phi\in \pi$, then $m=I(\phi)$.
\vskip 0.5cm
$2^o.$ We conclude: $m>0.$ For otherwise, $I(\phi)=0.$ However, we assumed that $G$ is connected, this implies $\phi=\phi_0=\frac{1}{d}\textbf{1}$. But, $\phi_0\notin \pi$. This is a contradiction.
\vskip 0.5cm
$3^o.$ According to lemma 5.11, we obtain a path $\gamma_1$ connecting $\phi$ and $\phi_0$ in $I_m\cap X$.
\vskip 0.5cm
Similarly, we have paths $\gamma_2$ connecting $\phi_0$ and $-\phi$, $\gamma_3$ connecting $-\phi$ and $-\phi_0$, $\gamma_4$ connecting $-\phi_0$ and $\phi$ in $I_m\cap X$.
Let
$$ A_0=\gamma_1 \circ \gamma_2 \circ \gamma_3 \circ \gamma_4 .$$
We have
$A_0\subset I_m\cap X$, and $A_0\simeq S^1.$ Thus $\gamma(A_0)\ge 2$, and then
$$ c_2=\inf_{\gamma(A)\ge 2} \sup_{x\in A} \tilde{I}(x)\le m,$$
i.e., $c_2\le m.$
\end{proof}
\vskip 0.5cm

From this theorem, one may compute the second eigenvalue by minimizing the function $I$ over the subset $\pi$ of $X$.

As an application of Theorem 5.8,  the following sufficient conditions for $\mu_2<1$ is presented.

\begin{theorem}
Let $G$ be a connected graph. If there are two groups of vertices $\{\alpha_1, \cdots, \alpha_k\},\,  \{\beta_1, \cdots, \beta_l\}$ satisfying:
\begin{enumerate}
\item $\Sigma^k_{i=1}d_{\alpha_i}= \Sigma^l_{j=1}d_{\beta_j},$ denoted by $c$,
\item In one of the two groups, there is at least a pair of adjacent vertices.
\end{enumerate}
Then $0<\mu_2\le 1-\frac{1}{2c}$.
\end{theorem}
\begin{proof}
Let us define
$$ x= \frac{1}{2c}(\Sigma^k_{i=1}\textbf{e}_{\alpha_i}- \Sigma^l_{j=1}\textbf{e}_{\beta_j}), $$
where $c=\Sigma^k_{i=1}d_{\alpha_i}.$
Obviously, $x\in X,$ and $\Sigma m_\alpha x_\alpha d_{i_\alpha}=0,$ i.e., $x\in \pi$. Assume without loss of generality:
$\alpha_1\sim \alpha_2.$ Let $D^0$ be the nodal set of $x$. We have
$$ I(x)=\Sigma_{j\sim i} |x_i-x_j|=\frac{1}{2c}[\Sigma_{\beta_j\sim \alpha_i} 2+\Sigma_{\beta_j\sim D^0} 1
+\Sigma_{\alpha_i\sim D^0} 1].$$
The total adjacent numbers starting from $\{\alpha_1, \cdots, \alpha_k\}$ or from $\{\beta_1, \cdots, \beta_l\}$ are at most $c$. But, at least one of the adjacent pairs is between $\alpha_1$ and $\alpha_2$, which does not have contribution in the summation. Thus
$$ \Sigma_{\beta_j\sim \alpha_i} 2+\Sigma_{\beta_j\sim D^0} 1
+\Sigma_{\alpha_i\sim D^0} 1\le 2c-1.$$
Following Theorem 5.12, we obtain
$$ \mu_2\le I(x)\le 1-\frac{1}{2c}.$$
\end{proof}
\vskip 0.5cm

\subsection{Cheeger's constant}

Given a graph $G=(V, E)$ and a subset of vertices $S \subset V$, the volume of $S$ is defined to be
$$ Vol(S)=\sum_{i\in S} d_i.$$
Let $\bar{S}=V\backslash S$. The edge boundary of $S$ is
$$ \partial S=\{ e=(i,j)\in E\,|\,\mbox{either}\,i\in S, j\notin S,\,\,\mbox{or} \,\,j\in S, i\notin S\}.$$
Thus,
$$\partial S=\partial \bar{S}.$$
It is denoted by $E(S, \bar{S})$.

The number
$$ h(G)=Min_{S}\frac{|E(S, \bar{S})|}{min(Vol(S), Vol(\bar{S}))}$$
is called Cheeger's constant\cite{refCe}, where $|S|$ is the cardinal number of $S$.
\vskip 0.5cm
The following lemma is adapted from the proof of Theorem 2.9 in Chung\cite{refCh}.
\begin{lemma} There exists a vector $y=(y_1, \cdots y_n)\in R^n$ such that
$$ h(G)=sup_{c\in R^1} \frac{\Sigma_{j\sim i}{|y_i-y_j|}}{\Sigma_{i=1}^n|y_i-c|d_i}.$$
\end{lemma}
\begin{proof}
By definition there exists a subset $S\subset V$ such that
$$h(G)=\frac{E(S, \bar{S})}{Vol(S)}, \,\,Vol(S)\le Vol(\bar{S}).$$
Define
\begin{equation*}y_i=\left\{\begin{array} {l}
1,\,\,\,\,\,\mbox{if}\,i\in S, \\
-1,\,\,\,\,\,\mbox{if}\,i\notin S.
\end{array}\right.
\end{equation*}
$i=1, \cdots, n$, the vector $y=(y_1, \cdots, y_n)$ is what we need. Indeed,
\begin{align*}
&sup_{c\in R^1} \frac{\Sigma_{j\sim i}{|y_i-y_j|}}{\Sigma_{i=1}^n|y_i-c|d_i}\\
&=Max_{|c|\le 1} \frac{\Sigma_{j\sim i}{|y_i-y_j|}}{\Sigma_{i=1}^n|y_i-c|d_i}\\
&=Max_{|c|\le 1} \frac{2 |E(S, \bar{S})|}{(1-c)Vol(S)+(1+c)Vol(\bar{S})}\\
&=\frac{|E(S, \bar{S})|}{Vol(S)}=h(G).
\end{align*}
\end{proof}
\vskip 0.5cm

Moreover, $h(G)$ has the following Minimax characterization:
$$ h(G)=inf_{x\in \mathbb{R}^n\backslash\{t\bar{1}\,|\,t\in R^1\}} sup_{c\in R^1} \frac{\Sigma_{j\sim i}{|x_i-x_j|}}{\Sigma_{i=1}^n|x_i-c|d_i}.$$
Cf. Theorem 2.9 Chung\cite{refCh}.

Applying the above results, we establish the relationship between $h(G)$ and $\mu_2.$ Namely,
\begin{theorem} Assume that $G=(V, E)$ is connected, then
$$     \mu_2 = h(G).$$
\end{theorem}
\begin{proof}
$1^0$ We prove: $ h(G)\ge \mu_2.$ From the above lemma, there exists a vector $y=(y_1, \cdots, y_n)$ such that
$$ h(G)=sup_{c\in R^1}\frac{\Sigma_{j\sim i}|y_i-y_j|}{\Sigma^n_{k=1}|y_k-c|d_k},$$
while
$$\mu_2= Min_{z\in \pi} \Sigma_{j\sim i}|z_i-z_j|.$$
Now let us define $t\in R^1$ such that
$$ \Sigma_{y_i<t}d_i\le \Sigma_{y_j\ge t}d_j,$$
$$ \Sigma_{y_i\le t}d_i\ge \Sigma_{y_j> t}d_j,$$
and let
$w=y-t \bar{1}$, where $\bar{1}=(1, \cdots, 1)$, then we have
$$|\delta^+(w)-\delta^-(w)|\le \delta^0(w).$$
Let $z_i=\frac{w_i}{\Sigma^n_{k=1}d_k|w_k|},\,i=1\cdots, n,\,\, z=(z_1, \cdots, z_n),$
We have $z\in \pi.$ and
$$ sup_{c\in R^1}\frac{\Sigma_{j\sim i}|y_i-y_j|}{\Sigma^n_{k=1}|y_k-c|d_k}\ge\frac{\Sigma_{j\sim i} |y_i-y_j|}{\Sigma^n_{k=1}|y_k-t|d_k}=\frac{\Sigma_{j\sim i}|w_i-w_j|}{\sum^n_{k=1} |w_k|d_k}=\sum_{j\sim i}|z_i-z_j|,$$
Therefore
$$h(G)\ge Min_{z\in \pi} \sum_{j\sim i}|z_i-z_j| =\mu_2.$$

$2^o$ Now, we turn to prove $h(G)\le \mu_2.$  To this end, we take the normalized second eigenvector $\phi$ into consideration. Let
$\phi=(x_1, x_2, \cdots, x_n)$, where $x_i=\nu$ times $\pm 1$ or $0, \,i=1,2, \cdots, n,$ and $\nu$ is a constant satisfying:
\begin{align*}
&\Sigma^n_{i=1} d_i |x_i|=1,\,\,\,i.e., (\delta^+ +\delta^-)\nu=1,\\
&|\delta^+ -\delta^-|\le \delta^0.
\end{align*}
Let $c_0$ be the minimum of $\Sigma^n_{i=1}d_i |x_i-c|$.
For $c_0\ge 0$,
\begin{align*}
&\Sigma^n_{i=1}d_i |x_i-c_0|\\
&=\delta^+(\nu-c_0)+ \delta^-(\nu+c_0)+\delta^0c_0\\
&=\nu(\delta_+ + \delta_-)+c_0(\delta^0-(\delta_+ - \delta_-))\\
&\ge \nu(\delta_+ + \delta_-)=1.\\
\end{align*}
For $c_0\le 0$, we also have
\begin{align*}
&\Sigma^n_{i=1}d_i |x_i-c_0|\\
&=\delta^+(\nu+|c_0|)+ \delta^-(\nu-|c_0|)+\delta^0|c_0|\\
&=\nu(\delta_+ + \delta_-)+|c_0|(\delta^0+(\delta_+ - \delta_-))\\
&\ge \nu(\delta_+ + \delta_-)=1.\\
\end{align*}
In summary,
$$ \inf_{c\in R^1}\sum^n_{i=1} d_i |x_i-c|\ge 1.$$
Thus
\begin{align*}
&\mu_2=\Sigma_{j\sim i} |x_i-x_j|\\
&\ge  sup_{c\in R^1}\frac{\Sigma_{j\sim i}|x_i-x_j|}{\Sigma^n_{i=1}|x_i-c|d_i}\\
&\ge inf_{y\in \mathbb{R}^n\backslash\{t\bar{1}\,|\,t\in R^1\}}sup_{c\in R^1}\frac{\Sigma_{j\sim i}|y_i-y_j|}{\Sigma^n_{i=1}d_i |y_i-c|}\\
&=h(G).
\end{align*}
We have proved $h(G)=\mu_2.$
\end{proof}
\vskip 0.5cm

\textbf{Remark}

In the linear spectral graph theory, it has been proved [9]:
$$ \frac{\lambda_2}{2}\le h(G)\le \sqrt{2\lambda_2}.$$

However, Cheeger's constant for a connected graph $G$ is exactly the second eigenvalue of $\Delta_1(G)$. This is the motivation of our study of the nonlinear eigenvalue theory.
In fact, a result similar to Theorem 5.15 has been given in Hein and Buehler \cite{refHB} for the unweighted $L^1$-norm 1-Laplacian,  where the volume function is the cardinality of the set. In a slightly different context in Hein and Setzer \cite{refHS} the Cheeger cut is discussed.

\vskip 0.5cm
As an example, we consider the Petersen graph:
$G=(V, E)$, where $V=\{1,2 \cdots, 10\},$ and
$$ E=\{ e_{12},e_{23},e_{34},e_{45},e_{51},e_{16},e_{27},e_{38},e_{49},e_{5\,10},e_{68},e_{69},e_{79},e_{7 \,10},e_{8\,10}\}.$$
Thus $d_1=d_2=\cdots =d_{10}=3.$
\vskip 0.5cm
It is easily verified:
\begin{align*}
&\mu_{2}=\frac{1}{3},\\
&\phi=\frac{1}{15}(1,1,1,1,1,-1,-1,-1,-1,-1),\\
&z_{12}=z_{23}=z_{34}=z_{45}=z_{51}=0,\\
&z_{68}=z_{69}=z_{7\, 10}=z_{8\, 10}=z_{79}=0,\\
&z_{16}=z_{27}=z_{38}=z_{49}=z_{5\, 10}=1,\
\end{align*}
By taking $S=\{1,2,3,4 ,5\}, \bar{S}=\{6,7,8,9,10\}$, we have
$$ E(S, \bar{S})=5, \,\,\,\,\, vol(S)=vol(\bar{S})=15,\,\,\,\,\,h(G)=\frac{1}{3}.$$
While $\lambda_2=\frac{2}{3}$, see [1].
\vskip 0.5cm

\section{1-Laplacian Spectral for Some Special Graphs}

In this section, we study the spectral of 1-Laplacian for some special graphs.
\subsection{$P_n$}
 A path with $n$ vertices is a graph of a sequence of $n$ vertices, starting from $1$ and ending at $n$ such that consecutive vertices are adjacent. It is denoted by $P_n$. In this graph
$$ d_1=d_n=1, \, d_2=d_3=\cdots=d_{n-1}=2.$$
Let $\phi=\Sigma^n_{i=1}x_i \textbf{e}_i$ be an eigenvector, with eigenvalue $\mu$. Then they satisfy the system:
\begin{equation*}(6.1)\left\{\begin{array} {l}
z_{12}\in \mu  Sgn (x_1),\\
z_{23}-z_{12}\in 2\mu Sgn (x_2)\\
\cdots,  \cdots\\
z_{n-1, n}-z_{n-2, n-1}\in 2\mu Sgn (x_{n-1})\\
-z_{n-1 n}\in \mu Sgn (x_n).\,\,\,\,\,\,\,\,\,\,\,\,\,\,
\end{array}\right.
\end{equation*}
It is known that $\mu=0$ is simple, with eigenvector $\frac{1}{2(n-1)} \textbf{1}$. We turn to case $\mu\in (0,1)$.

$1^o.$ We claim: $x_1\neq 0.$ For otherwise, $x_1=0$ implies $z_{12}=0,$ i.e., $0\in Sgn (x_1-x_2)$, it must be $x_2=0$. Repeating the deduction, it follows,
$$ x_n=x_{n-1}=\cdots =x_1=0.$$
This is impossible.

$2^o.$ No loss of generality, we assume $x_1>0.$ From the first equation
$$ \mu=z_{12}\in Sgn(x_1-x_2),$$
It must be $x_2=x_1$, and then by the second equation
$$z_{23}-z_{12}=2\mu.$$
This implies
$$ 3\mu\in Sgn (x_2-x_3).$$
Thus $x_3\le x_2$.

If $x_3<x_2$, then $\mu=\frac{1}{3}.$ Otherwise, $x_3=x_2,$ we repeat the procedure.

On the other side, we start from $x_n$, by the same procedure, and conclude: $0\neq x_n=x_{n-1.}$ Either
$\mu=\frac{1}{3}$ or $x_{n-2}\le x_{n-1},$ etc.
We define
\begin{equation*}\phi_{k+1}=\left\{\begin{array} {l}
\frac{1}{2(n-2k+1)}\Sigma^{r-k+1}_{i=1}(\textbf{e}_i-\textbf{e}_{n-i+1}),\,\,\,\,n=2r,\\
\frac{1}{2(n-2k)}\Sigma^{r-k+1}_{i=1}(\textbf{e}_i-\textbf{e}_{n-i+1}),\,\, \,\,\,n=2r+1.
\end{array}\right.
\end{equation*}
\begin{equation*}\mu_{k+1}=\left\{\begin{array} {l}
\frac{1}{n-2k+1},\,\,n=2r,\\
\frac{1}{(n-2k)}, \,\,\,n=2r+1.
\end{array}\right.
\end{equation*}
$k=1,2,\cdots, r-1.$ It is easy to verify that they satisfy the system (6.1).

\vskip 0.5cm
Finally, we study the case $\mu=1.$ Following lemma 5.1, each nodal domain consists of a single vertex. This means that no two consecutive vertices having the same sign.  Let
$$\phi_{r+k+1}=
\frac{1}{2(2k+1)}[\Sigma^{k}_{i=1}(\textbf{e}_{2i-1}-\textbf{e}_{2i})+(\textbf{e}_{2k+1}-\textbf{e}_n)],$$
$k=0,1,2,\cdots, r-1,$ either $n=2r$ or $n=2r+1,$ but in case $n=2r+1,$ we add
$$\phi_n=\frac{1}{4}(-\textbf{e}_1 + \textbf{e}_{r+1}-\textbf{e}_n).$$
Then,
\begin{equation*}sgn (x_j)=\left\{\begin{array} {l}
(-1)^{j-1},\,\,j=1,2, \cdots, 2k+1,\\
0, \,\,\,j=2k+1, \cdots, 2r-1,\\
-1, \,\,\,\,\,j=n
\end{array}\right.
\end{equation*}
By definition,
$$ z_{j, j+1}(\phi_{r+k+1})=(-1)^{j-1}\,\,\,\,\,j=1,2, \cdots, 2k+1.$$
$$ z_{n-1, n}(\phi_{r+k+1})=1$$
They satisfy the system (6.1).
\vskip 0.5cm

Now we have the conclusion:

\begin{equation*}n=2r:\left\{\begin{array} {l}
\mu_1=0,\,\,\,\,\,\,\,\,\phi_1=\frac{1}{2(n-1)}(1,\cdots, 1, 1\cdots,1)\\
\mu_2=\frac{1}{n-1},\,\,\phi_2=\frac{1}{2(n-1)}(1,\cdots, 1,-1\cdots ,-1)\\
\cdots,  \cdots\\
\mu_{k+1}=\frac{1}{n-2k+1},\,\phi_{k+1}(1,\cdots,1,0,\cdots,0,-1, \cdots,-1),\\
k=1,2, \cdots r-1\\
\mu_{r+k}=1,\,\,\,\phi_{r+k}=\frac{1}{2(k+1)}(1,-1,\cdots,-1,1,0,\cdots, 0, -1)\\
k=0,1,2,\cdots r-1.
\end{array}\right.
\end{equation*}
\vskip 0.5cm

\begin{equation*}n=2r+1:\left\{\begin{array} {l}
\mu_1=0,\,\,\,\,\,\,\,\,\phi_1=\frac{1}{2(n-1)}(1,\cdots, 1, 1\cdots,1)\\
\mu_2=\frac{1}{n-2},\,\,\phi_2=\frac{1}{2(n-1)}(1,\cdots, 1,0,-1\cdots ,-1)\\
\cdots,  \cdots\\
\mu_{k+1}=\frac{1}{n-2k},\,\phi_{k+1}(1,\cdots,1,0,\cdots,0,-1, \cdots,-1),\\
k=1,2, \cdots r-1\\
\mu_{r+k}=1,\,\,\,\phi_{r+k}=(1,-1,\cdots,-1,1,0,\cdots, 0, -1)\\
k=0,1,2,\cdots r.\\
\mu_n=1,\,\,\,\,\,\phi_n=(-1,0, \cdots,0,1,0,\cdots, 0,-1).
\end{array}\right.
\end{equation*}

\vskip 0.5cm

\subsection{$C_n$}
A cycle with $n$ vertices is a connected graph, where every vertex has exactly two neighbors. It is denoted by $C_n$. In this graph
$$ d_1=d_2=\cdots=d_{n-1}=d_n=2.$$
The eigenpair system reads as:
\begin{equation*}(6.2)\left\{\begin{array} {l}
z_{12}-z_{n1}\in 2\mu Sgn (x_1),\\
z_{23}-z_{12}\in 2\mu Sgn (x_2),\\
\cdots,  \cdots\\
z_{n-1, n}-z_{n-2, n-1}\in 2\mu Sgn (x_{n-1}),\\
z_{n1}-z_{n-1 n}\in 2 \mu Sgn (x_n).
\end{array}\right.
\end{equation*}
Obviously, $\mu=0$ is simple, with eigenvector $\frac{1}{2n} \textbf{1}$.
\vskip 0.5cm

As for $\mu\in (0,1)$, either $n=2r$ or $n=2r+1$, we have the following:

\begin{equation*}\left\{\begin{array} {l}
\mu_{k+1}=\frac{1}{r-k+1},\\
\phi_{k+1}= \frac{1}{4(r-k+1)} (1,\cdots,1,0,\cdots,0,-1, \cdots,-1),\\
k=1,2, \cdots r-1,\, \mbox{there are}\, 2(k-1)\,\mbox{zeros, if}\, n=2r;\,\, 2k-1\, \mbox{zeros, if}\, n=2r+1. \\
\mu_{r+k}=1,\,\,\,\phi_{r+k}=(0\cdots,0,1,-1,\cdots,1,-1)\\
k=1,2,\cdots r, \,\mbox{there are}\, n-2k\, \mbox{zeros}. \\
\end{array}\right.
\end{equation*}
Since the graph is cyclic invariant, all eigenvectors after cyclic transformation are eigenvectors.

The spectrum of $C_n$ is $\{0, \frac{1}{r}, \cdots, \frac{1}{2}, 1\}.$

\vskip 0.5cm
\subsection{$K_n$}

A graph $G$ is called complete, if any two vertices are adjacent. A complete graph with $n$ vertices is denoted by $K_n$. In this case,
$$ d_1=d_2= \cdots, =d_n=n-1.$$
Since any two vertices are adjacent, The  possible numbers for nodal domains are $r^+=r^-=1$. Thus a normalized eigenvector is of the form
$$ x=\delta^{-1}(\Sigma_{i\in D^+}-\Sigma_{i\in D^-}) \textbf{e}_i$$
For $\mu\neq 0, $ let
$$ card(D^+)=card(D^-)=k,$$
we obtain:
$$ \delta^{\pm}=\Sigma_{i\in D^{\pm}} d_i=(n-1)k,$$
and $card(D^0)=n-2k,$ either $n=2r$, or $n=2r+1$, $k=1,2, \cdots, r$.
\vskip 0.5cm

The graph $K_n$ is invariant under permutation group $\textit{S}_n$, we write down the eigenvectors without specifying the coordinates of indices.
\begin{equation*}n=2r\left\{\begin{array} {l}
\mu_{k+2}=\frac{r+k}{n-1},\\
\phi_{k+2}= \frac{1}{(n-2k)(n-1)} (1,\cdots,1,0,\cdots,0,-1, \cdots,-1),\\
k=0, 1,2, \cdots r-2,\, \mbox{there are}\,2k\,\mbox{zeros}\\
\mu_{r+k}=1,\\
\phi_{r+k}=(0\cdots,0,1,-1,\cdots,1,-1)\\
k=1,2,\cdots r, \,\mbox{there are}\, n-2k\, \mbox{zeros}. \\
\end{array}\right.
\end{equation*}
and
\begin{equation*}n=2r+1\left\{\begin{array} {l}
\mu_{k+2}=\frac{r+k+1}{n-1},\\
\phi_{k+2}= \frac{1}{(n-2k-1)(n-1)} (1,\cdots,1,0,\cdots,0,-1, \cdots,-1),\\
k=0,1,2, \cdots r-2,\, \mbox{there are}\, 2k+1\,\mbox{zeros}\\
\mu_{r+k}=1,\\
\phi_{r+k}=\frac{1}{2k(n-1)}(0\cdots,0,1,-1,\cdots,1,-1)\\
k=1,2,\cdots r, \,\mbox{there are}\, n-2k\, \mbox{zeros}. \\
\end{array}\right.
\end{equation*}
The spectrum for $K_n$ is $\{0, \frac{n}{2(n-1)}, \cdots, \frac{n-2}{n-1}, 1\}$ if $n$ is even, and
$\{0, \frac{n+1}{2(n-1)}, \cdots, \frac{n-2}{n-1}, 1\}$ if $n$ is odd.

\vskip 0.5cm

\textbf{Remark}

Amghibech \cite{refAm} provides some explicit examples for the eigenvalues of the p-Laplacian for $p>1$ similar to the above results.
\vskip 0.5cm

\textbf{Acknowledgement}

I am grateful to the referees for their valuable comments and suggestions. Especially, I want to thank one of the referees for drawing my attention to several earlier results and providing the related articles, which have been added to the references. I am particularly grateful to Prof. J. Y. Shao, who carefully read the manuscript and made many useful comments, and would like to thank Dr. S. M. Liu and Mr. D. Zhang for their help in preparing the paper.

The paper is supported by Chinese NSF No. 61121002 and  No. 11371038.

\bibliographystyle{amsplain}

\end{document}